\documentclass[prodmode,acmtcps]{acmsmall} % Aptara syntax

\usepackage[ruled]{algorithm2e}

\SetAlFnt{\small}
\SetAlCapFnt{\small}
\SetAlCapNameFnt{\small}
\SetAlCapHSkip{0pt}
\IncMargin{-\parindent}

\usepackage{tikz}
\usepackage{circuitikz}
\usepackage{balance}
\usepackage{color}
\usepackage{enumerate}
\usepackage{caption}
\usepackage{subcaption}
\usepackage{cite}
\usepackage{float}
\usepackage{empheq}
\usepackage{mathrsfs}
\usepackage{mathtools}
\usepackage{graphicx,amssymb,amstext,amsmath}
\usepackage[bookmarks=true]{hyperref}

\newcommand{\yong}[1]{{\color{black} #1}}

\makeatletter
\pgfcircdeclarebipole{}{\ctikzvalof{bipoles/interr/height 2}}{spst}{\ctikzvalof{bipoles/interr/height}}{\ctikzvalof{bipoles/interr/width}}{

    \pgfsetlinewidth{\pgfkeysvalueof{/tikz/circuitikz/bipoles/thickness}\pgfstartlinewidth}

    \pgfpathmoveto{\pgfpoint{\pgf@circ@res@left}{0pt}}
    \pgfpathlineto{\pgfpoint{.6\pgf@circ@res@right}{\pgf@circ@res@up}}
    \pgfusepath{draw}   
}

\def\pgf@circ@spst@path#1{\pgf@circ@bipole@path{spst}{#1}}
\tikzset{switch/.style = {\circuitikzbasekey, /tikz/to path=\pgf@circ@spst@path, l=#1}}
\tikzset{spst/.style = {switch = #1}}
\makeatother

\newtheorem{problem}{Problem}

\begin{document}

% Page heads
\markboth{S.Z. Yong et al.}{Switching and Data Injection Attacks on Stochastic Cyber-Physical Systems}

% Title portion
\title{Switching and Data Injection Attacks on Stochastic Cyber-Physical Systems: Modeling, Resilient Estimation and Attack Mitigation}
\author{Sze Zheng Yong
\affil{Arizona State University}
Minghui Zhu
\affil{Pennsylvania State University}
Emilio Frazzoli
\affil{Swiss Federal Institute of Technology (ETH)}
}
% NOTE! Affiliations placed here should be for the institution where the
%       BULK of the research was done. If the author has gone to a new
%       institution, before publication, the (above) affiliation should NOT be changed.
%       The authors 'current' address may be given in the "Author's addresses:" block (below).
%       So for example, Mr. Abdelzaher, the bulk of the research was done at UIUC, and he is
%       currently affiliated with NASA.

\begin{abstract}
{\sectionfont \textbf{ABSTRACT}}\\[1em]
In this paper, we consider the problem of attack-resilient state estimation, that is to reliably estimate the true system states despite two classes of attacks:
 %attacks. Specifically, we simultaneously consider two security issues for cyber-physical systems and the internet of things, namely 
(i) attacks on the switching mechanisms and (ii) false data injection attacks on actuator and sensor signals, in the presence of unbounded stochastic process and measurement noise signals. We model the systems under attack as hidden mode stochastic switched linear systems with unknown inputs and propose the use of a multiple-model inference algorithm to tackle these security issues. 
Moreover, we characterize fundamental limitations to resilient estimation (e.g., upper bound on the number of tolerable signal attacks) and discuss the topics of attack detection, identification and mitigation under this framework. 
Simulation examples of switching and false data injection attacks on a benchmark system and \yong{an IEEE 68-bus test system} show the efficacy of our approach to recover resilient (i.e., asymptotically unbiased) state estimates \yong{as well as to identify and mitigate the attacks}. 
\end{abstract}

\begin{bottomstuff}
This work was supported by the National Science Foundation, grant \#1239182. M. Zhu is partially supported by ARO W911NF-13-1-0421 (MURI) and NSF CNS-1505664.

Author's addresses: S.Z. Yong, School for Engineering of Matter, Transport and Energy, Arizona State University, Tempe, AZ (e-mail: szyong@asu.edu); M. Zhu,
School of Electrical Engineering and Computer Science, Pennsylvania State University, University Park, PA (e-mail: muz16@psu.edu); E. Frazzoli, Institute for Dynamic Systems and Control, Swiss Federal Institute of Technology (ETH), Z\"urich, Switzerland (email: efrazzoli@ethz.ch).
\end{bottomstuff}

\maketitle

\section{Introduction}

Cyber-physical systems (CPS) are systems in which computational and communication elements collaborate to control physical entities, and for networked CPS, the Internet of Things (IoT) interlinks these physical and cyber worlds %. The IoT keeps the different components of CPS 
in a continuous and close interaction. The cyber-physical coupling introduces new functions to control systems and improves their performance. However, control systems are also exposed to new cyber vulnerabilities. %but at the same time, introduces vulnerabilities into the systems. 
Such systems, which include the power grid, autonomous vehicles, medical devices, etc, are usually \emph{safety-critical} and if compromised or malfunctioning, can cause serious harm to the controlled physical entities and the people operating or utilizing them. Recent incidents of attacks on CPS, e.g., the Maroochy water breach, the StuxNet computer worm and various industrial security incidents \cite{Cardenas.2008b,Farwell.2011}, highlight a need for CPS and IoT security and for new designs of resilient estimation and control. 

\emph{\textbf{Literature review.}} 
Much of the early research focus has been on the characterization of undetectable attacks and on attack detection and identification techniques, which range from a simple application of data time-stamps in \cite{Zhu.2013} to \yong{anomaly detection methods using residuals (e.g., \cite{Mo.2010,Weimer.2012,kwon2013security}) with empirically chosen thresholds to trade-off between false alarms and probability of anomaly/attack detection. On the other hand, attack mitigation is typically considered from two perspectives---preventive and reactive \cite{combita2015response}---where preventive mitigation
identifies and removes system vulnerabilities to
prevent exploitation (e.g., \cite{dan2010stealth}) while reactive attack mitigation initiates countermeasures after detecting an attack and is
mainly studied using game-theoretic methods (e.g., \cite{ma2013markov,zhu2015game}).}

However, the ability to reliably estimate the true system states despite attacks (i.e., resilient state estimates) is just as desirable, if not more than purely attack detection or attack mitigation; thus, this problem has garnered considerable interest in recent years because the availability of resilient state estimates would, among others, allow for continued operation with the same controllers as in the case without attacks or for locational marginal pricing of electricity based on the real unbiased state information despite attacks. \yong{This problem has been studied both in the context of static systems (e.g., \cite{liu2011false,kosut2011malicious}) as well as dynamic systems as in this paper.}

For deterministic linear \yong{dynamic} systems under actuator and sensor signal attacks (e.g., via false data injection \cite{Cardenas.2008,Mo.2010,Pasqualetti.2013}), the resilient state estimation problem has been mapped onto an $\ell_0$ optimization problem that is NP-hard \cite{Pasqualetti.2013,Fawzi.2014}; thus, a relaxation to a convex problem is considered in \cite{Fawzi.2014}. Further extensions \cite{Pajic.2014,Pajic.2015} compute the worst-case bound on the state estimate error in the presence of additive noise errors with known bounds, while \cite{Yong.Foo.ea.ACC16} considers the resilient state estimation problem that is robust to bounded multiplicative and additive modeling and noise errors. However, these approaches do not apply in the presence of additive stochastic (unbounded) 
noise signals, which is one of the security issues we consider in this paper. On the other hand, \cite{Mishra.2015} %\cite{Mattingley.2010,Mishra.2015} 
consider systems with stochastic noise signals but with only sensor attacks. 

In addition, attacks that exploit the switching vulnerability of CPS and IoT or that alter its network topology have been recently identified as a serious CPS security concern. Some instances of such vulnerability are attacks on the circuit breakers of a smart grid \cite{Liu.2013} or on the logic mode (e.g., failsafe mode) of a traffic infrastructure \cite{Ghena.2014}, on the meter/sensor data network topology \cite{Kim.2013} and on the power system network topology \cite{Weimer.2012}. However, to the best of our knowledge, no resilient state estimators for dynamic systems have been developed to deal with this new class of attacks. 

Our techniques are based on simultaneous input and state estimation (see, e.g., \cite{Gillijns.2007,Gillijns.2007b,Yong.Zhu.ea.Automatica16}), where data injection attack vectors can be modeled as unknown inputs of dynamical systems. %Hence, a set of relevant literature pertains to that of SISE (e.g., \cite{Gillijns.2007,Yong.Zhu.Frazzoli.2013,Yong.Zhu.ea.Automatica16}).  }
Of particular importance to our approach are the stability and optimality properties as well as their relationship to strong detectability \cite{Yong.Zhu.ea.Automatica16}. Inspired by the multiple-model approach (see, e.g.,  \cite{Bar-Shalom.2002,Mazor.1998} and references therein), our previous work \cite{Yong.Zhu.ea.SICON16} introduced an inference algorithm that estimates hidden modes, unknown inputs and states simultaneously, which we now propose as the key tool to achieve resilient estimation. 

\emph{\textbf{Contributions.}} 
In this paper, we introduce a resilient state estimation algorithm that outputs reliable estimates of the true system states despite two classes of attacks. To our best knowledge, our resilient estimation algorithm is the first that addresses switching attacks as well as the first that successfully deals with simultaneous actuator and sensor attacks in the presence of unbounded stochastic noise signals. %In particular, we address some relatively unexplored issues} in ensuring secure estimation of cyber-physical systems that are interconnected by the internet of things: (i) switching or mode attacks (attacks of switching mechanisms altering system- and data-level network topologies), \yong{(ii) actuator and sensor signal attacks, and (iii) presence of unbounded stochastic process and measurement noise signals.} 

\yong{Our approach is built upon a general purpose inference algorithm developed in our previous work \cite{Yong.Zhu.ea.SICON16} for hidden mode stochastic switched linear systems with unknown inputs. 
The first novelty of the present paper lies in the modeling of switching and false data injection attacks on cyber-physical systems in the presence of unbounded stochastic noise signals as an instance  of this system class. In doing so, we show that unbiased state estimates (i.e., resilient state estimates) can be asymptotically recovered with the algorithm in \cite{Yong.Zhu.ea.SICON16}. Secondly,} we characterize fundamental limitations to resilient estimation \yong{that is useful for preventative mitigation}, such as the upper bound on the number of correctable/tolerable attacks, 
and consider the subject of attack detection. In addition, we provide sufficient conditions for designing unidentifiable attacks (from the attacker's perspective) and also sufficient conditions to obtain resilient state estimates even when the attacks are not identified (from the system operator/defender's perspective). Finally, we design an attack-mitigating and stabilizing feedback controller \yong{that contributes to the literature on non-game-theoretic reactive attack mitigation}. 

A preliminary version of this paper was presented in \cite{Yong.Zhu.ea.CDC15_resilient}, and this paper expands on those results  and includes new sections on attack detection and identification, as well as attack mitigation.  %which focused on consistency and fundamental limitation of attack-resilient estimation.  

\emph{\textbf{Paper Organization.}} %The remainder of this paper is organized as follows. 
Section 2 provides a motivating example of switching and data injection attacks on a multi-area power system. In Section 3, we describe the modeling of switching and false data injection attacks on cyber-physical systems and state our assumptions/models of the system and attacker. \yong{Section 4.1 reviews the multiple-model algorithm and its nice properties from \cite{Yong.Zhu.ea.SICON16} and provides an interpretation of the general purpose algorithm in the context of resilient state estimation. The rest of Section 4 is dedicated to the novel study of fundamental limitations to attack resilience.} Next, we focus on attack detection and identification in Section 5, and provide some sufficient conditions as guidelines for system operators/designers, while we design an attack-mitigating feedback controller in Section 6. Section 7 then demonstrates the effectiveness of our proposed approach on a benchmark system and \yong{an IEEE 68-bus test system}. %in  two simulation examples, including with the IEEE 68-bus test system. 
Finally, we conclude with some remarks in Section 8.

% Head 1
\section{Motivating Example} \label{sec:motivation}

To motivate the problem of resilient state estimation of stochastic cyber-physical systems under switching and false data injection attacks, let us consider an example of a power system % \cite{Wood.2013} 
with 3 control areas, each consisting of generators and loads, with transmission/tie-lines providing interconnections between areas  (see Figure \ref{fig:powerSys}). %A simplified model of the control areas and the tie-lines is given by (see also parameter definitions in  \cite[Chap. 10]{Wood.2013}):\vspace{0.15cm}

%\noindent {Control area $i$: }($i \in \{1,2,\hdots,N_{ca}\}$)
%\begin{align}
%\begin{array}{rl} 
%\displaystyle \frac{d \Delta \omega_i}{dt}+\frac{D_i \Delta \omega_i}{M_i} -\frac{ \Delta P_{mech_i}}{M_i}+\frac{\sum_{j \neq i} \Delta P^{ij}_{tie}}{M_i}&= \displaystyle-\frac{\Delta P_{L_i}}{M_i},\\
%\displaystyle\frac{d \Delta P_{mech_i}}{dt}+\frac{\Delta P_{mech_i}}{T_{CH_i}}-\frac{\Delta P_{v_i}}{T_{CH_i}}&=\displaystyle0,\\
%\displaystyle\frac{d \Delta P_{v_i}}{dt}+\frac{\Delta P_{v_i}}{T_{G_i}} +\frac{\Delta \omega_i}{R^f_i T_{G_i}}&= \displaystyle \frac{\Delta P_{ref_i}}{T_{G_i}};\end{array}
%\end{align}
%\noindent {Tie-line power flow, $P^{ij}_{tie}$, between areas $i$ and $j$ (no attack):}
%\begin{align}
%\begin{array}{rl}
%\displaystyle\frac{d \Delta P^{ij}_{tie}}{dt}&= \displaystyle T_{ij} (\Delta \omega_i - \Delta \omega_j),\\
%\displaystyle \Delta P^{ji}_{tie} &=\displaystyle - \Delta P^{ij}_{tie},\end{array}
%\end{align}
%
%
%\noindent where $\Delta \omega_i$, $\Delta P_{mech_i}$ and $\Delta P{v_i}$ represent deviations of the angular frequency, mechanical power and steam-valve position from their nominal operating values.

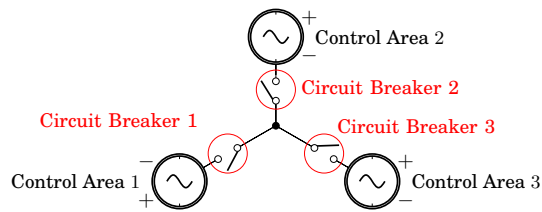
\begin{figure}[!b]
\centering
\resizebox{0.525\textwidth}{!}{%
\begin{circuitikz}[ scale =0.925,american voltages]
%\centering
 \draw 
(0,0) to[switch,-o] (0,1.3)
 to[short,-o](0,0.8)
(0,0) to[short,-o](0,0.45)
 (0,0) to[switch,*-] (-1.7321,-1) 
  to[short,-o](-1.0393,-0.6)
  (0,0) to[short,-o](-0.6928,-0.4)
 (0,0)  to[switch,-o] (1.7321,-1)
 to[short,-o](1.0393,-0.6)
  (0,0) to[short,-o](0.6928,-0.4);
 \draw[red] (0.866,-0.5) circle(10pt);
  \draw[red] (-0.866,-0.5) circle(10pt);
   \draw[red] (0,0.65) circle(10pt);
  \draw[thick,fill=white] (-1.7321,-1) circle(14pt);
 \draw (-1.7321,-1.4) to [sV=\small{Control Area $1$ \,}] (-1.7321,-0.6);
  \draw[thick,fill=white] (0,1.6) circle(14pt);
 \draw (0,2.0) to [sV=\small \, Control Area $2$] (0,1.2);
  \draw[thick,fill=white] (1.7321,-1) circle(14pt);
 \draw (1.7321,-0.6) to [sV=\small \, Control Area $3$] (1.7321,-1.4);
 \node[text width=3.5cm,red] at (-2.35,0.1) 
    {\small Circuit Breaker 1};
     \node[text width=3.5cm,red] at (3.0,0.0) 
    {\small Circuit Breaker 3};
      \node[text width=3.5cm,red] at (2.35,0.7) 
    {\small Circuit Breaker 2};
\end{circuitikz}
}
\caption{A 3-area power system with radial topology (corresponding to \emph{node/bus attack)}.\label{fig:powerSys}}
\end{figure}

A malicious agent is assumed to have access to circuit breakers that control the tie-lines (similar to \cite{Liu.2013}), and is thus able to sever the connection between control areas. Depending on the topology of the tie-line interconnection graph, such attacks may correspond to a \emph{node/vertex/bus attack} (disconnection of a control area from all others) or a \emph{link/edge/line attack} (disabling of a specific tie-line between two control areas), i.e., the power flow across the tie lines is altered. %: \vspace{0.15cm}
%\noindent {Attack on circuit breaker $i$ (node/bus attack):}\vspace{-0.05cm} 
%\begin{align}
%\Delta P^{i j}_{tie}=-\Delta P^{j i}_{tie}=0, \ \forall j \neq i;
%\end{align}
%\noindent {Attack on circuit breaker $(i,j)$ (link/line attack):} 
%\begin{align}
%\Delta P^{i j}_{tie}=-\Delta P^{j i}_{tie}=0.
%\end{align}
Moreover, we assume that the system dynamics and state measurements are subject to random noise and attacks via additive false data injection in the actuator and sensor signals. 

The goal of resilient state estimation is thus to obtain unbiased state estimates despite switching attacks, i.e., attacks on switches/circuit breakers, and data injection attacks on actuators and sensors.

\section{Problem Formulation} \label{sec:Problem}

\subsection{Attack Modeling} \label{sec:attackModel}
We consider two different classes of possibly time-varying 
attacks on cyber-physical systems (CPS):
\begin{description}
\item[\textbf{Data Injection Attacks}] Attacks on actuator and sensor signals via manipulation or injection with ``false" signals of unknown \emph{magnitude} and \emph{location} (i.e., subset of attacked actuators or sensors). In other words, signal attacks consist of both \emph{signal magnitude attacks} and \emph{signal location attacks}.   \emph{Examples:} 
Denial of service, deceptive attacks via data injection \cite{Cardenas.2008,Pasqualetti.2013}.
\item[\textbf{Switching Attacks}] Attacks on the switching mechanism that changes the system's \emph{mode} of operation, or on the sensor data or interconnection network  \emph{topology}, which we will also refer to as \emph{mode attacks}. 
\emph{Examples:} Attack on circuit breakers \cite{Liu.2013}, power network topology \cite{Weimer.2012}, sensor data network \cite{Kim.2013} and logic switch of a traffic infrastructure \cite{Ghena.2014}. 

\end{description}

\subsubsection{Data Injection Attacks}
%To further illustrate data injection attacks, 
For clarity of exposition, we assume for the moment that there is only one mode of operation, and that the linear system dynamics is not perturbed by stochastic noise signals:
\begin{align*}
\begin{array}{l}
x_{k+1}=A_k x_k + B_k (u_k+d^a_k), \quad 
y_k = C_k x_k +D_k (u_k+d^a_k) + d^s_k,
\end{array}
\end{align*}
where $x_k \in \mathbb{R}^n$ is the continuous state, $y_k \in \mathbb{R}^\ell$ is the sensor output, $u_k \in \mathbb{R}^m$ is the known input, $d^a_k \in \mathbb{R}^m$ and $d^s_k \in \mathbb{R}^\ell$ are attack signals that are injected into the actuators and sensors, respectively. The attack signals are sparse, i.e., if sensor $i \in \{1,\cdots,\ell\}$ is not attacked then necessarily $d_{k}^{s,(i)}=0$ for all time steps $k$; otherwise $d_k^{s,(i)}$ can take any value. Since we do not know which sensor is attacked, we refer to this uncertainty as the \emph{signal location attack}, and the arbitrary values that $d_k^{s,(i)}$ can take as the \emph{signal magnitude attack}. The same observation holds for the attacks on actuators $d^a_k$.

If, in addition, we have knowledge of which the actuators and sensors are vulnerable to data injection attacks, we will incorporate this information using $\overline{G}_k$ and $\overline{H}_k$ to result in following system dynamics
\begin{align*}
\begin{array}{l}
x_{k+1}=A_k x_k + B_k u_k+ \overline{G}_k d^a_k, \quad 
y_k = C_k x_k +D_k u_k + \overline{D}_k d^a_k+ \overline{H}_k d^s_k.
\end{array}
\end{align*}
If no such knowledge is available, $\overline{G}_k=B_k$, $\overline{D}_k=D_k$ and $\overline{H}_k=I$. Moreover, in some cases, the actuator and sensor attack signals are known to be mixed and cannot be separated. In order to take this into consideration, we represent the potentially `mixed' attack signals with $d_k$ and introduce corresponding $G_k$ and $H_k$ matrices to obtain
\begin{align*}
\begin{array}{l}
x_{k+1}=A_k x_k + B_k u_k+ {G}_k d_k,\quad
y_k = C_k x_k +D_k u_k +{H}_k d_k.
\end{array}
\end{align*}
In the absence of mixed attack signals, $d_k= \begin{bmatrix} d^a_k & d^s_k \end{bmatrix}$, $G_k =\begin{bmatrix} \overline{G}_k & 0 \end{bmatrix}$ and $H_k =\begin{bmatrix} \overline{D}_k & \overline{H}_k \end{bmatrix}$. The description of these matrices will be made more precise in Section \ref{sec:sysDes}.

\subsubsection{Switching Attacks}
On the other hand, a system may have multiple modes of operation, denoted by the set $\mathcal{Q}^m$ of cardinality $t_m \triangleq |\mathcal{Q}^m|$, either through the presence of switching mechanisms or different configurations/topologies of the sensor data or interconnection network, i.e., each mode $q_k \in \mathcal{Q}^m$ has its corresponding set of system matrices, $\{A^{q_k}_k,B^{q_k}_k, C^{q_k}_k, D^{q_k}_k,G^{q_k}_k,  H^{q_k}_k\}$. A \emph{switching attack} or \emph{mode attack} then refers to the ability of an attacker to choose and change the mode of operation $q_k$ without the knowledge of the system operator/defender.

\paragraph{\textbf{Attacker Model Assumptions}}
We do not constrain the malicious \emph{signal magnitude attack} to be a signal of any type (random or strategic) nor to follow any model, thus no prior `useful' knowledge of the dynamics of $d_k$ is available (uncorrelated with $\{d_\ell\}$ for all $k\neq \ell$, $\{w_\ell\}$ and $\{v_\ell\}$ for all $\ell$). 
%The \emph{number} of actuators/sensors that are attacked, $p$, must be less than the maximum that is asymptotically correctable according to  and (ii) the switching mechanisms/topologies that may be compromised (hence, the number of possible modes of operation, $t_m$, when under mode attack), as well as (iii) that 
%However, \yong{we assume that} attackers can only switch their strategies for both mode and signal attacks \yong{after a sufficiently long time (i.e., after detection delays)}. Note that this assumption is reasonable and realistic if the time a malicious agent takes to regain access/control is large compared to the time scale for the convergence of the inference algorithm and/or when the intention of the malicious agent is to confuse or avoid detection through intermittent attacks.

\subsection{System Description} \label{sec:sysDes}

In this section, we take the perspective of a system operator/defender, i.e., as one with the goal of obtaining resilient/reliable state estimates. Thus, our techniques include the modeling of the system in a way that facilitates the design of a resilient state estimation algorithm. 
Since we now assume that the system is perturbed by random, unbounded process and measurement noise signals, we model the switching and false data injection attacks on a noisy dynamic system using a \emph{hidden mode switched linear discrete-time stochastic system with unknown inputs} \yong{(i.e., a dynamical system with multiple modes of operation where the system dynamics in each mode is linear and stochastic, and the mode and some inputs are not known/measured}; cf. Figure \ref{fig:hyb_diag}): %  \yong{for an illustration}): 
\begin{align} 
\nonumber (x_{k+1},{q}_k)\hspace{-0.05cm}&=( A_k^{q_k} x_k\hspace{-0.1cm}+\hspace{-0.1cm}B_k^{q_k} u^{q_k}_k\hspace{-0.1cm}+\hspace{-0.1cm}G_{k}^{q_k} d^{q_k}_{k} \hspace{-0.1cm}+\hspace{-0.1cm}w^{q_k}_k,q_k), \qquad%\\
x_k\in \mathcal{C}_{q_k},\\
(x_k,q_k)^+ \hspace{-0.05cm}&=(x_k,\delta^{q_k}(x_k)),  \quad \qquad \quad \qquad \quad \qquad \qquad x_k \in \mathcal{D}_{q_k}, \label{eq:system}\\
\nonumber y_k&=C^{q_k}_{k} x_k\hspace{-0.1cm} +\hspace{-0.1cm} D^{q_k}_{k} u^{q_k}_k \hspace{-0.1cm}+\hspace{-0.1cm}H^{q_k}_{k} d^{q_k}_k \hspace{-0.1cm}+\hspace{-0.1cm} v^{q_k}_{k}, 
\end{align}
where $x_k \in \mathbb{R}^n$ is the continuous system state and $q_k \in \mathcal{Q} = \{1,2,\hdots,\mathfrak{N}\}$ is the hidden discrete state or \emph{mode}, which a malicious attacker has access to, \yong{while $\mathcal{C}_{q_k}$ and $\mathcal{D}_{q_k}$ are flow and jump sets, and $\delta^{q_k}(x_k)$ is the mode transition function. For more details on the hybrid systems formalism, see \cite{goebel.sanfelice.ea.09}}. For each mode $q_k$, $u^{q_k}_k \in U_{q_k} \subset \mathbb{R}^m$ is the known input, $d^{q_k}_k \in \mathbb{R}^p$ the unknown input or \emph{attack signal} and $y_k \in \mathbb{R}^l$ the output,
whereas the corresponding process noise $w_k^{q_k} \in \mathbb{R}^n$ and measurement noise $v^{q_k}_k \in \mathbb{R}^l$ are mutually uncorrelated, zero-mean Gaussian white random signals with known covariance matrices, $Q^{q_k}_k=\mathbb{E} [w_k^{q_k} w_k^{q_k \top}] \succeq 0$ and $R^{q_k}_k=\mathbb{E} [v^{q_k}_k v_k^{q_k \top}] \succ 0$, respectively.  Moreover, $x_0$ is  independent of $v^{q_k}_k$ and $w^{q_k}_k$ for all $k$. %\yong{The remaining system matrices will be described in more detail and with some examples below.}

\begin{figure}[!t]
\begin{center}
\includegraphics[scale=0.375]{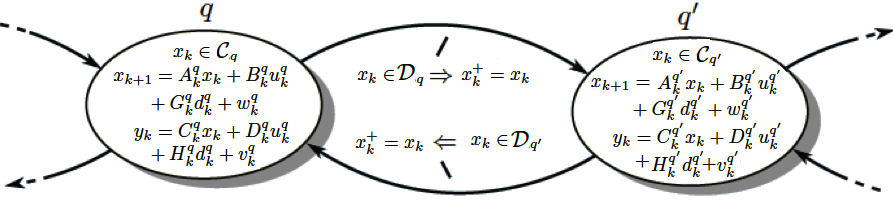}
\caption{Illustration of a switched linear system with unknown inputs as a hybrid automaton with two modes, $q$ and $q'$, \yong{where the system dynamics in each mode is linear}.\label{fig:hyb_diag} }
\end{center}
\end{figure}

Our stochastic cyber-physical system (CPS) model in \eqref{eq:system} is capable of  capturing the unknowns or uncertainties introduced by the switching and data injection attacks to the system of interest that are both \emph{categorical} and \emph{continuous}. The hidden mode allows us to model the categorical nature of the switching and data injection attacks (\emph{mode attack} and \emph{signal location attack}), whereas the unknown input captures the continuous nature of the \emph{signal magnitude attacks}.

At any particular time $k$, the stochastic CPS is in precisely one of its modes, which is not measured, hence \emph{hidden}. The following remark motivates the consideration of more modes than those corresponding to switching/mode attacks given by $\mathcal{Q}^m$.

\begin{remark}
Suppose again for simplicity that there is only one mode of operation, i.e., $\mathcal{Q}^m$ is a singleton. Then, in the ideal scenario for the system operator/defender that the system $(A_k,G_k,C_k,H_k)$ is strongly detectable, unbiased estimates of states $x_k$ can be obtained and the attack signal $d_k$ can also be identified \cite{Yong.Zhu.ea.Automatica16}. Unfortunately, this property does not hold in general. In fact, Theorem \ref{thm:max_errors} will reveal that we need a small number of vulnerable actuators and sensors to enable resilient state estimation. Thus, we will exploit the sparse nature of the false data injection attacks, and consider more models/modes in a set $\mathcal{Q}^d$, each with fewer vulnerable actuators and sensors to make sure that strong detectability holds. % (cf. Theorem \ref{thm:max_errors} for a possible explanation). 
\end{remark}

Thus, the modes we consider in the model set $\mathcal{Q} \triangleq \mathcal{Q}^m \times \mathcal{Q}^d$ (as described below), whose cardinality will be characterized in Theorem \ref{thm:maxModels}, include 
\begin{enumerate}[(i)]
\item the modes of operation, $\mathcal{Q}^m$, that attacked switching mechanisms (e.g., circuit breakers, relays) operate via access to the jump set $\mathcal{D}_{q_k}$ and the mode transition function $\delta^{q_k}(\cdot)$, or the possible interconnection network topologies that dictate the system matrices, $A^{q_k}_k$ and $B^{q_k}_k$, and the sensor data network topologies, $C^{q_k}_k$ and $D^{q_k}_k$, that an attacker can choose ({\emph{mode attack}}), as well as 
\item the different hypotheses for each mode, $\mathcal{Q}^d$, about which actuators and sensors are attacked or not attacked, represented by $G^{q_k}_k$ and $H^{q_k}_k$ ({\emph{signal location attack}}). 
\end{enumerate}

More precisely, we assume that $G^{q_k}_k\triangleq \mathcal{G}_k \mathcal{I}_G^{q_k}$ and $H^{q_k}_k\triangleq \mathcal{H}_k \mathcal{I}_H^{q_k}$  for some input matrices $\mathcal{G}_k \in \mathbb{R}^{n \times t_a}$ and $\mathcal{H}_k \in \mathbb{R}^{\ell \times t_s} $, where $t_a$ and  $t_s$ are the number of actuator and sensor signals that are \emph{vulnerable}, respectively. Note that $p_a^{q_k} \leq t_a \leq m$ and $p_s^{q_k} \leq t_s \leq \ell$, i.e., the number of \emph{attacked} actuator signals $p_a^{q_k}$ under mode/hypothesis $q_k$ cannot exceed the number of \emph{vulnerable} actuators and in turn cannot exceed the total number of actuators $m_a$. The same holds for $p_s^{q_k}$ \emph{attacked} sensors from $t_s$ \emph{vulnerable} sensors out of $\ell$ measurements. 
Moreover, we assume that the maximum total number of attacks is $p \triangleq p_a^{q_k} +p_s^{q_k}  \leq p^*$, where $p^*$ is the maximum number of asymptotically correctable signal attacks (cf. Theorem \ref{thm:max_errors} for its characterization).

On the other hand, $\mathcal{I}_G^{q_k} \in \mathbb{R}^{t_a \times p}$ and $\mathcal{I}_H^{q_k}  \in \mathbb{R}^{t_s \times p}$ are index matrices such that $d_k^{a,q_k}\triangleq \mathcal{I}_G^{q_k} d_k$ and $d_k^{s,q_k}\triangleq \mathcal{I}_H^{q_k} d_k$ represent the subvectors of $d_k \in \mathbb{R}^p$  representing \emph{signal magnitude attacks} on the actuators and sensors, respectively. 
These matrices provide a means of incorporating information about the way the attacks affect the system, e.g., when the same attack is injected to an actuator and a sensor, or if some states are \emph{not} attacked, according to a particular hypothesis/mode $q_k$ about the signal attack location.
It is noteworthy that our approach specifies which actuators and sensors are \emph{not attacked}, in contrast to the approach in \cite{Mishra.2015}, which removes \emph{attacked} sensor measurements but is not applicable for actuator attacks.

The following are some examples for choosing $\mathcal{G}_k$, $\mathcal{H}_k$, $\mathcal{I}_G^{q_k}$ and $\mathcal{I}_H^{q_k}$  to encode additional information about the nature/structure of data injection attacks.

\begin{example}
For a 2-state system with 2 vulnerable actuators and 1 vulnerable sensor, if the same attack signal is injected into the first actuator and the sensor under the hypothesis corresponding to mode $q_k$, then $\mathcal{G}_k=I_2$, %\begin{bmatrix} 1 & 0 \\ 0 & 1 \end{bmatrix}$,
 $\mathcal{H}_k=1$, $\mathcal{I}_G^{q_k}=I_2$ and $\mathcal{I}_H^{q_k}=\begin{bmatrix} 1 & 0 \end{bmatrix}$. \yong{In this case, we obtain}  $G_k^{q_k}=I_2$ and $H_k^{q_k}=\begin{bmatrix} 1  & 0 \end{bmatrix}$.
\end{example}

\begin{example}
For a 3-state system with 3 actuators and 2 sensors, if the first actuator and the second sensor are not vulnerable and there are 3 attacks according to the hypothesis  corresponding to mode $q_k$, then $\mathcal{H}_k=\begin{bmatrix} 1  \\ 0 \end{bmatrix}$,  $\mathcal{I}_G^{q_k}=\begin{bmatrix} 1 & 0 & 0 \\ 0 & 1 & 0 \end{bmatrix}$, $\mathcal{I}_H^{q_k}=\begin{bmatrix} 0 & 0 & 1\end{bmatrix}$ and $\mathcal{G}_k=\begin{bmatrix} 0 & 0 \\  1 & 0  \\ 0 & 1 \end{bmatrix}$. \yong{In this case, we have} $G_k^{q_k}=\begin{bmatrix} 0 & 0 & 0 \\ 1  & 0 & 0 \\ 0 & 1 & 0 \end{bmatrix}$ and $H_k^{q_k}=\begin{bmatrix} 0 & 0 & 1 \\ 0  & 0 & 0 \end{bmatrix}$.
\end{example}

\paragraph{\textbf{System Assumptions}} 
We require that the system is
\emph{strongly detectable}\footnote{A linear system is \emph{strongly detectable} if $y_k=0$  $\forall k \geq 0$ implies $x_k \to 0$ as $k \to \infty$
for all initial states $x_0$ and input sequences $\{d_i\} _{i\in \mathbb{N}}$ (see \cite[Section 3.2]{Yong.Zhu.ea.Automatica16} for necessary and sufficient conditions for this property). } 
in each mode. In fact, strong detectability is \emph{necessary} for each mode in order to asymptotically correct the unknown attack signals (also necessary for deterministic systems \cite[Theorem 6]{Sundaram.2007}). Note also that strongly detectable systems need not be stable (cf. example in the proof of Theorem \ref{thm:max_errors}), but rather that the strongly undetectable modes of such systems are stable. 
%\subsection{Assumptions on Attacker, Defender and System} \label{sec:assumptionsSA}
%
\vspace{-0.15cm}
\paragraph{\textbf{Knowledge of the System Operator/Defender}}
The matrices $A^{q_k}_k$, $B^{q_k}_k$, $G^{q_k}_k$, $C^{q_k}_k$, $D^{q_k}_k$ and $H^{q_k}_k$ are known, as well as the system assumption of strong detectability in each mode. Moreover, the only knowledge of the defender concerning the malicious attacker is about (i) the upper bound on the \emph{number} of actuators/sensors that can be attacked, $p$, and (ii) the switching mechanisms/topologies that may be compromised. The upper bound $p$ in the former assumption allows the defender, in the worst case, to enumerate all possible combinations of $G^{q_k}_k$ and $H^{q_k}_k$. On the other hand, the latter assumption allows the defender to consider all possible topologies/modes of operations.

Alternatively, the above assumptions on the system and attackers can be viewed as recommendations or guidelines for system designers/operators to secure their systems as a preventative attack mitigation measure. For instance, the requirement of strong detectability allows system designers to determine which actuators or sensors need to be safeguarded to guarantee resilient estimation. %, whereas the bounded switching frequency assumption suggests that mechanisms to prevent fast/frequent switchings may need to be designed.

\subsection{Security Problem Statement}

With the above characterization, the resilient state estimation problem is identical to the mode, state and input estimation problem, where the unknown inputs represent the unknown signal magnitude attacks and each mode/model represents an \emph{attack mode} (resulting from the unknown mode attacks and unknown signal attack locations). The \emph{objective} of this paper is:

\begin{problem}\label{prob}
Given a stochastic cyber-physical system described by \eqref{eq:system}, 
\begin{enumerate}
\item develop a \emph{resilient estimator} that asymptotically recovers \emph{unbiased} estimates of the system state and attack signal  irrespective of the location or magnitude of attacks on its actuators and sensors as well as switching mechanism/topology (mode) attacks, \label{prob1}
\item characterize fundamental limitations associated to the inference algorithm we developed, specifically the maximum number of asymptotically correctable signal attacks and the maximum number of required models with our multiple-model approach,  \label{prob2} 
\item study the conditions under which attacks can be detected or noticed 
%remain undetectable or can always be detected 
(attack detection) and under which the attack strategy can be identified (attack identification) using the resilient state estimator we developed, and \label{prob3}
\item design tools for attack mitigation via attack-rejection feedback control. \label{prob4}
\end{enumerate}
\end{problem}

\section{Resilient State Estimation}\label{sec:attack}

To achieve resilient state estimation against switching attacks in the presence of stochastic process and measurement noise signals, we have shown in the previous section that the system under switching and false data injection attacks is representable as a hidden mode, switched linear system with unknown inputs given in \eqref{eq:system}. Since we do not know the true model (i.e., the attack strategy corresponding to the true \emph{mode attack} and \emph{signal location attack}), combinations of possible attack strategies need to be considered, and as such, the multiple-model estimation approach is a natural choice for solving this problem. Thus, we propose the use of the \yong{general purpose} multiple-model algorithm that we previously designed and applied to vehicle collision avoidance \cite{Yong.Zhu.ea.SICON16} as our resilient state estimation algorithm for solving Problem 1.\ref{prob1}. 

We will begin with a brief summary of the multiple-model inference algorithm and its nice properties \cite{Yong.Zhu.ea.SICON16}. Then, we consider Problem 1.\ref{prob2} and characterize some fundamental limitations to resilient estimation in Section \ref{sec:fundLim}. 

\subsection{Resilient State Estimation Algorithm and Properties} \label{sec:prelim}
\subsubsection{Multiple-Model State and Input Filtering Algorithm}
The multiple-model (MM) approach we take is inspired by the multiple-model filtering algorithms for hidden mode hybrid systems with \emph{known} inputs (e.g.,  \cite{Bar-Shalom.2004,Mazor.1998} and references therein), that have been widely applied for target tracking. Our multiple-model framework (see Figure \ref{fig:multipleModel}) consists of the parallel implementation of a \emph{bank of input and state filters} \cite{Yong.Zhu.ea.Automatica16} with each model corresponding to a system mode \yong{(i.e., of \emph{mode-matched filters that simultaneously estimate states and unknown inputs from sensor measurements and known inputs for each mode})}. The objective of the MM approach is then to decide which model/mode is the best representation of the current system mode as well as to  estimate the state and unknown input of the system based on this decision.

\yong{In this subsection, we provide an \emph{abbreviated review}} of the multiple-model approach for simultaneous mode, state and unknown input estimation given in \cite{Yong.Zhu.ea.SICON16}. Two variants of the multiple-model inference algorithm---static and dynamic---were proposed in that work. The latter provides a possibility of incorporating prior knowledge about the switching strategy of the attack. However, we assume no such knowledge about the malicious agent; thus we consider only the static variant (cf. Algorithm \ref{algorithm2} and Figure \ref{fig:staticMM}), which consists of: (i) a bank of mode-matched filters, and (ii) a likelihood-based approach for computing model probability. \\[-0.15cm]

\begin{figure}[!b]
\begin{center}
\includegraphics[scale=0.285]{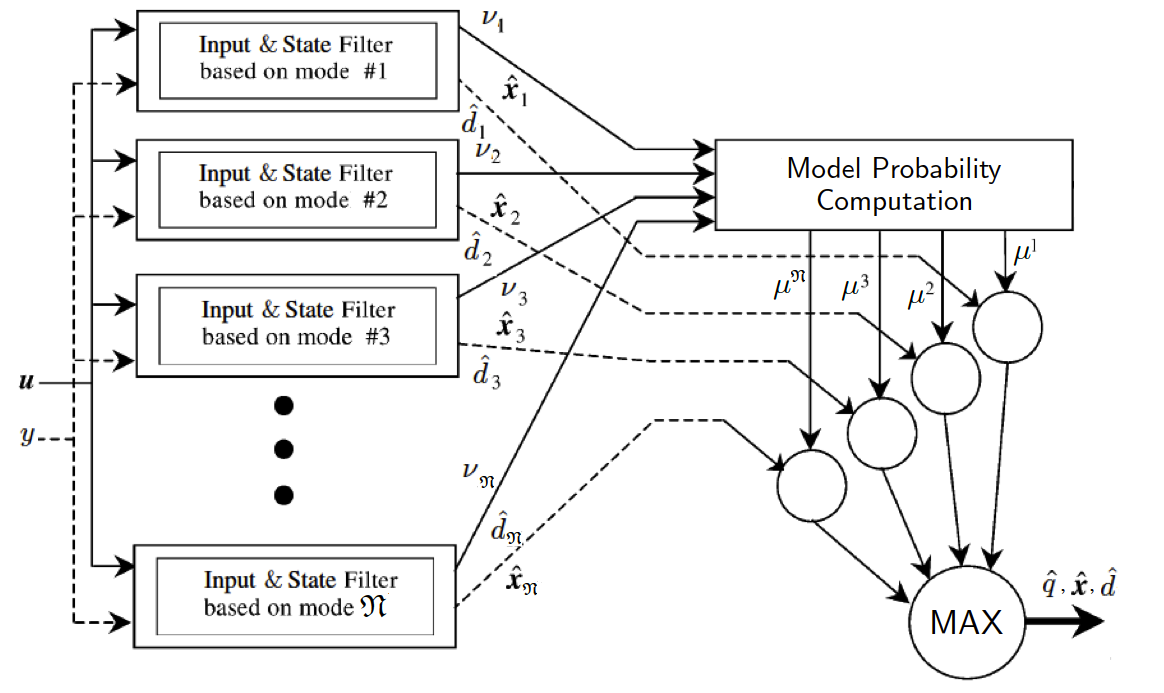}
\caption{Multiple-model framework for hidden mode, input and state estimation. \yong{Each model consists of a mode-matched filter that uses known inputs $u$ and outputs $y$ to estimate states $\hat{x}$ and unknown inputs $\hat{d}$, in addition to generalized innovations $\nu$ that in turn, determine the most probable mode $\hat{q}=\arg \max_{j \in \{1,\hdots,\mathfrak{N}\}} \mu^j$.}\label{fig:multipleModel} }
\end{center}
\end{figure}

\begin{figure}[!h]
\begin{center}
\includegraphics[scale=0.405]{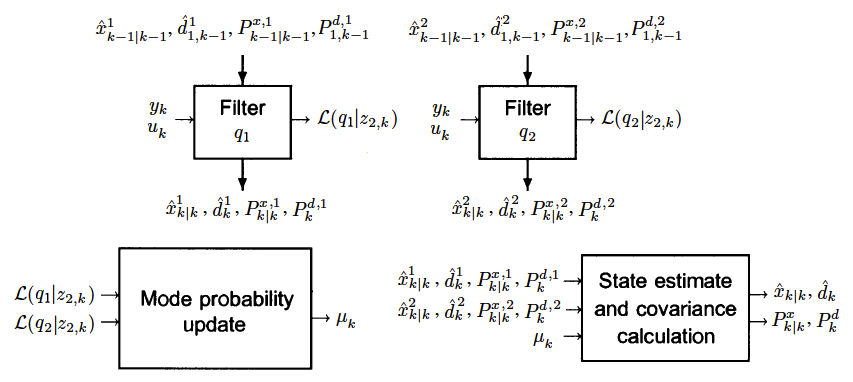}
\caption{Illustration of \yong{the exchange of signals/information between components of} a static multiple-model estimator with two %hidden modes, using two 
mode-matched input and state filters. % given by \eqref{eq:d}, \eqref{eq:xstar} and \eqref{eq:stateEst}.
\label{fig:staticMM} }
\end{center}
\end{figure}

\emph{\textbf{Mode-Matched Filters.}}
The bank of filters is comprised of $\mathfrak{N}$ simultaneous state and input filters, one for each mode, based on the optimal recursive filter developed in \cite{Yong.Zhu.ea.Automatica16} (superscript $q_k$ omitted to increase readability; cf. Algorithm \ref{algorithm1}). \vspace{0.1cm} 

\noindent\emph{Unknown Input Estimation}: %\vspace{-0.1cm} 
\begin{align}
\begin{array}{rl}
\hat{d}_{1,k} \hspace{-0.0cm}&=M_{1,k} (z_{1,k}-C_{1,k} \hat{x}_{k|k}-D_{1,k} u_k),\\
\hat{d}_{2,k-1}\hspace{-0.0cm}&=M_{2,k} (z_{2,k}-C_{2,k} \hat{x}_{k|k-1}-D_{2,k} u_k),\\
\hat{d}_{k-1}\hspace{-0.0cm}&= V_{1,k-1} \hat{d}_{1,k-1} + V_{2,k-1} \hat{d}_{2,k-1}, \end{array}\label{eq:d}
\end{align}
\emph{Time Update}:
\begin{align}
\hspace{-0.3cm}\begin{array}{rl}
 \hat{x}_{k|k-1}\hspace{-0.05cm}&=A_{k-1} \hat{x}_{k-1 | k-1} + B_{k-1} u_{k-1} + G_{1,k-1} \hat{d}_{1,k-1}, \\
\hat{x}^\star_{k|k}&=\hat{x}_{k|k-1}+G_{2,k-1} \hat{d}_{2,k-1}, \label{eq:xstar}
\end{array}\hspace{-0.3cm}
\end{align}
\emph{Measurement Update}:
\begin{align}
\hat{x}_{k|k}
&= \hat{x}^\star_{k|k} +\tilde{L}_k (z_{2,k}-C_{2,k} \hat{x}^\star_{k|k}-D_{2,k} u_k),  \quad \label{eq:stateEst}
\end{align}
where $\hat{x}_{k-1|k-1}$, $\hat{d}_{1,k-1}$, $\hat{d}_{2,k-1}$ and $\hat{d}_{k-1}$ denote the optimal estimates of $x_{k-1}$, $d_{1,k-1}$, ${d}_{2,k-1}$ and $d_{k-1}$. \yong{The remaining notations are best understood in the context of the system transformation described in Appendix \ref{sec:appendix}.} Due to space constraints, the filter derivation as well as necessary and sufficient conditions for filter stability and optimality are omitted; the reader is referred to \cite{Yong.Zhu.ea.Automatica16} for details. %a detailed discussion.\\[-0.1cm]

\emph{\textbf{Mode Probability Computation.}}
To compute the probability of each mode, the multiple-model approach exploits the whiteness property \cite[Theorem 1]{Yong.Zhu.ea.SICON16} of the \emph{generalized innovation} sequence, $\nu_k$, defined as
\begin{align}
&\nu_k \triangleq  \tilde{\Gamma}_k (z_{2,k}-C_{2,k} \hat{x}^\star_{k|k}-D_{2,k} u_k), %\triangleq \tilde{\Gamma}_k \overline{\nu}_k, 
\label{eq:g-inno} 
\end{align}
i.e., $\nu_k \sim \mathcal{N}(0,S_k)$ \yong{(a multivariate normal distribution)} with covariance $S_k\triangleq \mathbb{E}[\nu_k \nu_k^\top]=\tilde{\Gamma}_k \tilde{R}^\star_{2,k} \tilde{\Gamma}_k^\top$ and
where $\tilde{\Gamma}_k$ is chosen such that $S_k$ is invertible and $\tilde{R}^\star_{2,k}$ is given in Algorithm \ref{algorithm1}. \yong{In the context of resilient estimation, the generalized innovation represents the residual signal where the effects of false data injection attacks have been removed.
Then, using this ``attack-free'' generalized innovation,} we define the \emph{likelihood function} for each mode $q$ at time $k$ conditioned on all prior measurements $Z^{k-1}$:
\begin{align}
\mathcal{L}(q_k|z_{2,k})\triangleq %f_{z_{2,k}|Z^{k-1},q_k} (z_{2,k}|Z^{k-1},q_k)=f_{\nu_k|Z^{k-1},q_k} (\nu_k|Z^{k-1},q_k)=
 \mathcal{N}(\nu^{q_k}_k; 0, S_k^{q_k})\yong{=\frac{\exp(-\frac{1}{2}\nu^{q_k \, \top}_k (S_k^{q_k})^{-1}\nu^{q_k}_k)}{\sqrt{|2\pi S_k^{q_k}|}}. }
 \label{eq:likelihood} 
\end{align}
 Then, using Bayes' rule, the posterior probability \yong{$\mu_k^j$} for each mode $j$ is \yong{recursively computed from the prior probability $\mu_{k-1}^j$} using
\begin{align}
\mu^j_k 
= P(q_k=j|z_{1,k},z_{2,k},Z^{k-1}) = \frac{\mathcal{N}(\nu^j_k; 0, S_k^j) \mu^j_{k-1}}{\sum^\mathfrak{N}_{i=1} \mathcal{N}(\nu^i_k; 0, S_k^i) \mu^j_{k-1}}.
\label{eq:mu_j}
\end{align}
Note that a heuristic lower bound on all mode probabilities is imposed such that the  modes are kept alive in case of a switch in the attacker's strategy. Finally, based on the posterior mode probabilities, the most likely mode at each time $k$, \yong{$\hat{q}_k$,} and the associated state and input estimates and covariances, \yong{$\hat{x}_{k|k}$, $\hat{d}_{k}$, ${P}^x_{k|k}$ and ${P}^d_{k}$}, are determined: %, as follows: 
\begin{align}
\begin{array}{l}
\hat{q}_k  =j^*  =  \arg \max_{j \in \{1,\hdots,\mathfrak{N}\}}  \mu^j_k, \,
\hat{x}_{k|k}  = \hat{x}^{j^*}_{k|k},\, 
 \hat{d}_{k}= \hat{d}^{j^*}_{k},\, 
{P}^x_{k|k}= {P}^{x,j^*}_{k|k}, \, 
{P}^d_{k}= {P}^{d,j^*}_{k}. 
\end{array}\label{eq:outputcombi}
\end{align}

\begin{algorithm}[!bp]
\DontPrintSemicolon % Some LaTeX compilers require you to use \dontprintsemicolon instead
\KwIn{$q_k$, $\hat{x}_{k-1|k-1}^{q_k}$, $\hat{d}_{1,k-1}^{q_k}$, $P^{x,q_k}_{k-1|k-1}$, $P^{xd,q_k}_{1,k-1}$, $P^{d,q_k}_{1,k-1}$ [superscript $q_k$ omitted in the following]}
%\KwOut{$\tilde{R}^{\star,q_k}_{2,k}$, $\hat{x}^{\star,q_k}_{k|k}$}
$\triangleright$ Estimation of $d_{2,k-1}$ and $d_{k-1}$\;
 $\hat{A}_{k-1}=A_{k-1}-G_{1,k-1}M_{1,k-1} C_{1,k-1}$;\;
 $\hat{Q}_{k-1}=G_{1,k-1}M_{1,k-1}R_{1,k-1}M_{1,k-1}^\top G_{1,k-1}^\top +Q_{k-1}$;\;
$\tilde{P}_k=\hat{A}_{k-1} P^{x}_{k-1|k-1} \hat{A}_{k-1}^\top +\hat{Q}_{k-1}$;\;
$\tilde{R}_{2,k}=C_{2,k} \tilde{P}_k C_{2,k}^\top+R_{2,k}$;\;
$P^d_{2,k-1}=(G_{2,k-1}^\top C_{2,k}^\top \tilde{R}_{2,k}^{-1} C_{2,k} G_{2,k-1})^{-1}$;\;
$M_{2,k}=P^d_{2,k-1} G_{2,k-1}^\top C_{2,k}^\top \tilde{R}^{-1}_{2,k}$;\;
$\hat{x}_{k|k-1}=A_{k-1} \hat{x}_{k-1|k-1}+B_{k-1} u_{k-1}+G_{1,k-1} \hat{d}_{1,k-1}$;\;
$\hat{d}_{2,k-1}=M_{2,k} (z_{2,k}-C_{2,k} \hat{x}_{k|k-1}-D_{2,k} u_k)$;\;
$\hat{d}_{k-1} =V_{1,k-1} \hat{d}_{1,k-1} + V_{2,k-1} \hat{d}_{2,k-1}$; \;
$P^d_{12,k-1}=M_{1,k-1} C_{1,k-1} P^{x}_{k-1|k-1} A_{k-1}^\top C_{2,k}^\top M_{2,k}^\top-P^{d}_{1,k-1} G_{1,k-1}^\top C_{2,k}^\top M_{2,k}^\top$;\;
$P^d_{k-1}=V_{k-1} \begin{bmatrix} P^{d}_{1,k-1} & P^d_{12,k-1} \\ P^{d \top}_{12,k-1} & P^d_{2,k-1} \end{bmatrix} V_{k-1}^\top$;\;
$\triangleright$ Time update\;
$\hat{x}^\star_{k|k}=\hat{x}_{k|k-1}+G_{2,k-1} \hat{d}_{2,k-1}$;\;
$P^{\star x}_{k|k}=G_{2,k-1} M_{2,k} R_{2,k} M_{2,k}^\top G_{2,k}^\top+(I-G_{2,k-1}M_{2,k}C_{2,k})\tilde{P}_k(I-G_{2,k-1}M_{2,k}C_{2,k})^\top$;\;
$\tilde{R}^\star_{2,k}=C_{2,k} P^{\star x}_{k|k} C_{2,k}^\top +R_{2,k} -C_{2,k} G_{2,k-1} M_{2,k} R_{2,k}-R_{2,k} M_{2,k}^\top G_{2,k-1}^\top C_{2,k}$;\;
$\triangleright$ Measurement update\;
$\breve{P}_k=P^{\star x}_{k|k} C_{2,k}^\top - G_{2,k-1} M_{2,k} R_{2,k}$; \;
$\tilde{L}_k=\breve{P}_k\tilde{R}^{\star \dagger}_{2,k}$; \;
$\hat{x}_{k|k}=\hat{x}^\star_{k|k}+\tilde{L}_k(z_{2,k}-C_{2,k} \hat{x}^\star_{k|k}-D_{2,k} u_k)$;\;
$P^x_{k|k}=\tilde{L}_k R^\star_{2,k} \tilde{L}_k^\top- \tilde{L}_k \breve{P}_k^\top - \breve{P}_k \tilde{L}_k^\top$;\;
$\triangleright$ Estimation of $d_{1,k}$\;
$\tilde{R}_{1,k}=C_{1,k}P^x_{k|k}C_{1,k}^\top+R_{1,k}$; \;
$M_{1,k}=\Sigma_k^{-1}$; \;
$P^d_{1,k}=M_{1,k} \tilde{R}_{1,k} M_{1,k}$; \;
$\hat{d}_{1,k}=M_{1,k} (z_{1,k}-C_{1,k} \hat{x}_{k|k}-D_{1,k} u_k)$; \;
\Return{$\tilde{R}^{\star,q_k}_{2,k}$, $\hat{x}^{\star,q_k}_{k|k}$}\;
\caption{{\sc Opt-Filter} finds the optimal state and input estimates for mode $q_k$}
\label{algorithm1}
\end{algorithm}

\begin{algorithm}
\DontPrintSemicolon % Some LaTeX compilers require you to use \dontprintsemicolon instead
\KwIn{$\forall  j \in  \{1,2,\hdots,\mathfrak{N}\}$: $\hat{x}^j_{0|0}$; $\mu^j_0$; \; \hspace{1cm} $\hat{d}^j_{1,0} = (\Sigma_0^{j})^{-1} (z^j_{1,0}- C^j_{1,0} \hat{x}^j_{0|0} -D^j_{1,0} u_0)$; $P^{d,j}_{1,0}=(\Sigma_0^{j})^{-1}(C^j_{1,0} P^{x,j}_{0|0} C^{j \top}_{1,0}+R^j_{1,0})(\Sigma_0^{j})^{-1}$;}
%\KwOut{$\tilde{R}^{\star,q_k}_{2,k}$, $\hat{x}^{\star,q_k}_{k|k}$}
 \For {$k =1$ to $N$}{
  \For {$j =1$ to $\mathfrak{N}$}{
$\triangleright$ Mode-Matched Filtering\;
Run {\sc Opt-Filter}($j$,$\hat{x}_{k-1|k-1}^{j}$, $\hat{d}_{1,k-1}^{j}$, $P^{x,j}_{k-1|k-1},P^{d,j}_{1,k-1})$; \;
$\overline{\nu}^j_k\triangleq z^j_{2,k}-C^j_{2,k} \hat{x}^{\star,j}_{k|k}-D^j_{2,k} u_k$; \;
$\mathcal{L}(j|z^j_{2,k})=\frac{1}{(2\pi)^{p^j_{\tilde{R}}/2} |\tilde{R}^{j,\star}_{2,k}|_+^{1/2}}\exp \left(-\frac{\overline{\nu}_k^{j \top} \tilde{R}^{j,\star  \dagger}_{2,k}  \overline{\nu}^j_k }{2}\right)$; \;
}
\For {$j =1$ to $\mathfrak{N}$}{
$\triangleright$ Mode Probability Update (small $\epsilon>0$)\;
 $\overline{\mu}^j_k=\max\{\mathcal{L}(j|z^j_{2,k}) \mu_{k-1}^j,\epsilon\}$; \;
}
\For {$j =1$ to $\mathfrak{N}$}{
$\triangleright$ Mode Probability Update (normalization)\;
$\mu^j_k=\frac{\overline{\mu}^j_k}{\sum^\mathfrak{N}_{\ell=1} \overline{\mu}^\ell_k}$; \;
$\triangleright$ {Output}\;
Compute \eqref{eq:outputcombi}; \;
}
}
\Return{$\hat{x}_{k|k}$, $P^{x}_{k|k}$ }\;
\caption{{\sc Resilient State Estimator (Static-MM-Estimator)} finds resilient state estimates corresponding to most likely mode}
\label{algorithm2}
\end{algorithm}

\subsubsection{Properties of the Resilient State Estimator}\label{sec:ident}

Our previous work \cite{Yong.Zhu.ea.SICON16} shows that the resilient state estimator has nice asymptotic properties, namely \yong{(i) \emph{mean consistency}, i.e., the geometric mean of the mode probability for the true model $\ast \in \mathcal{Q}$ asymptotically converges to 1 for all initial mode probabilities \cite[Theorem 8]{Yong.Zhu.ea.SICON16} and (ii) \emph{asymptotic optimality}, i.e., the state and input estimates in \eqref{eq:outputcombi} converge on average 
 to optimal state and input estimates in the minimum variance unbiased sense \cite[Corollary 13]{Yong.Zhu.ea.SICON16}.} % that we will briefly describe without proofs. % in the following. 

\subsection{Fundamental Limitations of Attack-Resilient Estimation} \label{sec:fundLim}

Next, we consider Problem 1.\ref{prob2} and characterize fundamental limitations of the attack-resilient estimation problem and of our multiple mode filtering approach, \yong{which constitutes one of the main results in this paper}. First, assuming for the moment that there is only one mode of operation (no switching attacks), we will upper bound the number of asymptotically correctable signal attacks/errors (i.e., signal attacks whose effects can be asymptotically negated or cancelled such that unbiased state estimates are still available). 
Then, we provide the maximum number of models that is required by our multiple-model approach to obtain resilient estimates despite %switching and data injection 
attacks. 

\subsubsection{Number of Asymptotically Correctable Signal Attacks}
More formally, we introduce the following definition for only data injection attacks, which in itself is an interesting research problem in the CPS security community.  %\vspace{-0.1cm}
\begin{definition}[Asymptotically/Exponentially Correctable Signal Attacks]\label{def:1}
We say that $p$ actuators and sensors signal attacks are asymptotically/exponentially correctable, if for any initial state $x_0 \in \mathbb{R}^n$ and signal attack sequence $\{ d_j\}_{j \in \mathbb{N}}$ in $\mathbb{R}^p$, we have an estimator such that the estimate bias asymptotically/exponentially tends to zero, i.e., $\mathbb{E}[\hat{x}_k-x_k] \to 0$ (and $\mathbb{E}[\hat{d}_{k-1}-d_{k-1}] \to 0$) as $k \to \infty$. 
\end{definition}%\vspace{-0.25cm}

\begin{remark} \label{rem:1}
Note the distinction in the definitions of asymptotically/exponentially correctable signal attacks in Definition \ref{def:1} and of correctable signal attacks in \cite[Definition 1]{Fawzi.2014}. Their definition implies finite-time estimation and is related to strong observability \cite{Fawzi.2014}. 
Due to the new challenges of further considering stochastic noise signals and mode switching, we adopt the weaker notion of asymptotic estimation, which only requires a `weaker' condition of strong detectability (implied by strong observability \cite{Yong.Zhu.ea.Automatica16}). 
This is mainly for the sake of theoretical analysis. Simulation results demonstrate that our algorithm has practically finite-time convergence.
\end{remark} 

To derive an estimation-theoretic upper bound on the maximum number of signal attacks that can be asymptotically corrected, we assume 
that the true model or mode ($q_k=*$) is known. 
Thus, the resilient state estimation problem is identical to the state and input estimation problem in \cite{Yong.Zhu.ea.Automatica16}, where the unknown inputs represent the attacks on the actuator and sensor signals. It has been shown in \cite{Yong.Zhu.ea.Automatica16} that unbiased states (and also unknown inputs) can be obtained asymptotically (exponentially) only if the system is strongly detectable (cf. \cite{Yong.Zhu.ea.Automatica16} for more details, e.g. regarding filter stability and existence). With this in mind, the upper bound on the maximum number of signal attacks that can be asymptotically (exponentially) corrected is given as follows: 

\begin{theorem}[Maximum Correctable Data Injection Attacks] \label{thm:max_errors}
The maximum number of asymptotically (exponentially fast) correctable actuators and sensors signal attacks, $p^*$, for system \eqref{eq:system}  
is equal to the number of sensors, $l$, i.e., $p^* \leq l$ and the upper bound is achievable. %\vspace{-0.15cm}
\end{theorem}
\begin{proof} 
A necessary and sufficient condition for strong detectability (with the true model $q_k=*$) 
is given in \cite{Yong.Zhu.ea.Automatica16} as
\begin{align} \label{eq:strongDet}
{\rm rk}\begin{bmatrix} zI-A^\ast & -G^\ast \\ C^\ast & H^\ast \end{bmatrix}=n+p^\ast, \ \forall z \in \mathbb{C}, |z| \geq 1.
\end{align}
Since the above system matrix has only $n+l$ rows, it follows that its rank is at most $n+l$. Thus, from the necessary condition for \eqref{eq:strongDet}, we obtain $n+p^* \leq n+l \Rightarrow p^*\leq l$. We show that the upper bound is achievable using the example of the discrete-time equivalent model (with time step $\Delta t=0.1 s$) of the smart grid case study in \cite{Liu.2013}, where in both circuit breaker modes, $A=\begin{bmatrix}  \ \ 0.9520 &   0.0936 \\   -0.9358  &  0.8584 \end{bmatrix}$ and $G=\begin{bmatrix} 0 \\ 0 \end{bmatrix}$. If the first state is measured but compromised (e.g., $C=\begin{bmatrix} 1 & 0 \end{bmatrix}$ and $H=1$ $\Rightarrow$ $p^\ast=l$), it can be verified that the system is strongly detectable, i.e., with two invariant zeros at $\{0.9945 \pm 0.0311j\}$ that are strictly in the unit circle in the complex plane. 
Similarly, it can be verified that the unstable system with matrices $A=\begin{bmatrix} 1.5 & 1 \\ 0  & 0.1 \end{bmatrix} $, $G=\begin{bmatrix} 1 & 0 \\ 0  & 0 \end{bmatrix}$, $C=\begin{bmatrix} 1 & 0 \\ 0  & 1 \end{bmatrix}$ and $H=\begin{bmatrix} 0 & 0 \\ 0  & 1 \end{bmatrix}$ (i.e., with $p^*=l$) has an invariant zero at $\{0.1\}$ and is hence strongly detectable. Thus, in both cases, the optimal filter in \cite{Yong.Zhu.ea.Automatica16} can be applied and unbiased state estimates can be asymptotically achieved when $p^\ast=l$.
\end{proof} 

The theorem above implies that for each mode of operation that results from switching attacks, the total number of vulnerable actuators and sensors must not exceed the number of measurements. Moreover, it is worth reiterating that the necessity of strong detectability can serve as a guide to determine which actuators or sensors need to be safeguarded to guarantee resilient estimation, for preventative attack mitigation. Since strong detectability is a system property that is independent
of the filter design, the necessity of this property can be viewed as a fundamental limitation for resilient estimation, i.e., the ability to asymptotically/exponentially obtain
unbiased estimates.

\subsubsection{Number of Required Models for Estimation Resilience}\label{sec:numberModels}

Then, in a similar spirit as the attack set identification approach of \cite{Weimer.2012,Pasqualetti.2013} in which a bank of deterministic residuals are computed to determine the true attack set (but not the magnitude of the attacks), we consider a bank of filters to find the most probable model/mode. We now characterize the maximum  number of models $\mathfrak{N}^\ast$ that need to be considered with the multiple-model approach in Section \ref{sec:prelim} \yong{(which is independent of the size of the system, e.g., the number of buses in a power system)}: 

\begin{theorem}[Maximum Number of Models/Modes]  \label{thm:maxModels}
Suppose there are $t_a$ actuators and $t_s$ sensors, and at most $p \leq l$ of these signals are attacked. Suppose also that there are $t_m$ possible attack modes (\emph{mode attack}). Then, the combinatorial number of all possible models, and hence the maximum number of models that need to be considered with the multiple-model approach, is $$\mathfrak{N}^\ast=t_m \binom{t_a+t_s}{p}=t_m \binom{t_a+t_s}{t_a+t_s-p}.$$ 
\end{theorem}
\begin{proof}
It is sufficient to consider only models corresponding to the maximum number of attacks $p$. All models with strictly less than $p$ attacks are contained in this set of models with the attack vectors having some identically zero elements for which our estimation algorithm is still applicable. Thus, we only need to consider combinations of $p$ attacks among $t_a+t_s$ sensors and actuators for each of the $t_m$ attack modes of operation/topologies. Note that this number is the maximum because resilience may be achievable with less models: For instance, when $t_m=1$, $t_a=0$ and $t_s=2=l$, $p=1$, $A=\begin{bmatrix} 0.1 & 1 \\ 0 & 0.2 \end{bmatrix}$ and $C=I_2$, we have $\mathfrak{N}^\ast=2$, but it can be verified that with $G=0_{2 \times 2}$ and $H=I_2$ (only one model, i.e., $1=\mathfrak{N}<\mathfrak{N}^\ast$), the system is strongly detectable.
\end{proof} %\vspace{-0.25cm}

\begin{remark}
If $\mathfrak{N} >1$, the multiple-model approach requires that the number of attacks is strictly less than the number of sensor measurements, i.e., $p<l$. Otherwise, the generalized innovation \eqref{eq:g-inno} is empty and we have no means of selecting the `best' model, i.e., of computing mode probabilities.
\end{remark}

We now discuss how the availability of additional knowledge about the data injection attack strategies may influence the number of models that needs to be considered in relation to the number of models $\mathfrak{N}^*$ in Theorem \ref{thm:maxModels} when such knowledge is not available. %In brief, we observed that the availability of such knowledge may increase or decrease the necessary number of models. On one hand,} 
%If more information about the attacks is known, then one may expect  that less models need to be considered. For instance, 
Suppose we have additional knowledge that there are at most $n_a \leq t_a$ and $n_s \leq t_s$ attacks on the actuators and sensors, respectively, with a total of $p$ attacks (where $p \leq l$ and $p \leq n_a+n_s$), then the maximum number of models that are required, $$ \mathfrak{N}^\ast=t_m \sum_{i=0}^{\min\{n_a,p\}} \binom{t_a}{i} \binom{t_s}{\min\{p-i,n_s\}},$$
is less than the number required in combinatorial case in Theorem \ref{thm:maxModels}. 

On the other hand, the knowledge that less actuators or sensors are vulnerable may actually increase the number of models, as shown in the following example with $t_m=1$ (one mode of operation), $n_a=0$ (no attacks on actuators), $A=\begin{bmatrix} 0.1 & 1 \\ 0 & 1.2 \end{bmatrix}$ and $C=I$. Suppose only one of the two sensors is vulnerable, $n_s=p=1 < l=2$, then we have to consider 2 models with $G=\begin{bmatrix}0\\ 0 \end{bmatrix}$, $H_1=\begin{bmatrix} 1 \\ 0 \end{bmatrix}$ and $H_2=\begin{bmatrix} 0 \\ 1\end{bmatrix}$. On the other hand, if both sensors are vulnerable $n_s=p=2$, then only one model is required with $G=0$ and $H=I$. Note, however, that the latter case is not strongly detectable with zeros at $\{0.1,1.2\}$, thus this system violates the necessary condition in  \cite{Yong.Zhu.ea.Automatica16} for obtaining resilient estimates; but both systems in the former case can be verified to be strongly detectable. 
%Nonetheless, with more information, 
%\yong{However,} the problem with more attacks that was previously not solvable because the (fewer) models are not strongly detectable, may now become solvable because although the number of models is increased, each of these models is strongly detectable. %An example of this is with $t_m=1$, $A=\begin{bmatrix} 0.1 & 1 \\ 0 & 1.2 \end{bmatrix}$ and $C=I$. If we assume that $n_a=0$ and $n_s=p=2$, then with $G=0$ and $H=I$ (only one model is required), the system is not strongly detectable with zeros at $\{0.1,1.2\}$. However, if $n_a=0$ and $n_s=p=1 < l=2$, we have 2 models with $G=\begin{bmatrix}0\\ 0 \end{bmatrix}$, $H_1=\begin{bmatrix} 1 \\ 0 \end{bmatrix}$ and $H_2=\begin{bmatrix} 0 \\ 1\end{bmatrix}$, where both models can be verified to be strongly detectable.

\section{Attack Detection and Identification}

In this section, we consider Problem 1.\ref{prob3} and study the consequence of the asymptotic properties of the resilient state estimation algorithm (static MM filter) in Section \ref{sec:ident} on attack detection and identification. %Then, we will consider the attack mitigation problem by designing an attack-rejection feedback controller \yong{in Section \ref{sec:mitigation}}. 

%In this section, we discuss some conditions under which the attacker's strategy and actions can be uniquely identified (cf. definition below).

First, note that the resilient state estimation algorithm we presented in the previous section is oblivious to whether the switching and false data injection attacks on the system are strategic. Nonetheless, we would like to understand how strategic attacks can be detected or identified by our algorithm. Specifically, we consider strategic attackers whose goal is to choose data injection signals $d_k$ and the true mode $* \in \mathcal{Q}$ in order to mislead the system operator/defender into believing that the mode of operation is $q \in \mathcal{Q}, q \neq *$. If such an attack action cannot be reconstructed/identified by the system operated, then we refer to this attack as \emph{unidentifiable}. If, in addition, the attack is not noticeable, then this attack is \emph{undetectable}, formally defined as follows: %If  when the system operator uses the resilient estimator in Algorithm \ref{algorithm2}, then we refer to these attacks as \yong{\emph{unidentifiable}}, formally defined as follows:

%\begin{definition}[Undetectable Switching and Data Injection Attack] \label{def:undetectable} Under the normal mode of operation, $q \in \mathcal{Q}$, the combination of a mode attack and a data injection attack is \emph{undetectable}, if the attacked mode and signal attack locations corresponding to $* \in \mathcal{Q}, * \neq q$ and signal magnitude attacks $d_k$ are such that  $D(f^\ast_\ell \| f^q_\ell) = 0$ (i.e., Algorithm \ref{algorithm2} is not mean consistent according to Theorem \ref{thm:mean_cons}).
%\end{definition}

\begin{definition}[Switching and Data Injection Attack Detection] \label{def:detection} A switching and data injection attack is detected if %there exists a mechanism/oracle to tell that 
the true mode $* \in \mathcal{Q}$ (chosen by attacker) has the maximum mean probability when using the resilient state estimation algorithm in Algorithm \ref{algorithm2} or is not distinguishable from another mode $q \in \mathcal{Q}, q \neq *$ (chosen by defender) on average. %, which reveals the chosen \emph{mode attack} and \emph{signal attack location}, and asymptotically unbiased estimates of attack signals $d_k$ can be obtained, i.e., the \emph{signal magnitude attack} is reliably estimated. 
\end{definition}

\begin{definition}[Switching and Data Injection Attack Identification] \label{def:identification} A switching and data injection attack strategy is identified if the attack is detected and in addition, the true mode $* \in \mathcal{Q}$ is uniquely determined on average, which reveals the chosen \emph{mode attack} and \emph{signal attack location}, and asymptotically unbiased estimates of attack signals $d_k$ can be obtained, i.e., the \emph{signal magnitude attack} is reliably estimated. 
\end{definition}

From the above definitions, it is clear that if an attack is undetectable, then it will also be unidentifiable. On the flip side, if an attack is identifiable, then it is  detectable. Note, however, that attack detection or identification is not needed for obtaining resilient state estimates. 
%be detected or identified for obtaining resilient estimates.  It is noteworthy that the model consistency property discussed in Section \ref{sec:ident}, while is quintessential for establishing the soundness of the multiple-model approach, it may not be necessary for resilience. 
For instance, in the trivial case that there are no attacks $d_k=0$ for all $k$, the state estimates of all models would perform equally well. This means that the attacks need not be detected or identified for obtaining resilient estimates. 

\subsection{Attack Detection}

Fortunately, our resilient state estimation algorithm in Algorithm \ref{algorithm2} guarantees that an attack will always be detected by Definition \ref{def:detection}.
\begin{theorem}[Attack Detection] \label{thm:Detguarantee}
The resilient state estimation algorithm in Algorithm \ref{algorithm2} (with ratios of prior being identically 1) guarantees that switching and data injection attacks are always detectable.
\end{theorem}
\begin{proof}
Since the Kullback Leibler divergence $D(f^\ast_\ell \| f^q_\ell)$ is greater than or equal to 0 with equality if and only if $f^\ast_\ell=f^q_\ell$ (\hspace{-0.01cm}\cite[Lemma 3.1]{Kullback.1951}), with $j=\ast \in \mathcal{Q}$ as the true model and $i \in \mathcal{Q}, i \neq \ast$, the summand in the exponent of the ratio of geometric means whose expression is given in \cite[Lemma 14]{Yong.Zhu.ea.SICON16} %Lemma \ref{lem:ratios} 
is always non-negative, i.e., $D(f^\ast_\ell \| f^i_\ell)-D(f^\ast_\ell \| f^\ast_\ell)=D(f^\ast_\ell \| f^i_\ell) \geq 0$. 
In other words, the ratio of the true model mean probability to the model mean probabilities of any other mode ($i\in \mathcal{Q}, i \neq \ast$) cannot decrease and can at best remain the same as the ratio of their priors which is 1 by assumption. Thus, either the true model is identified or both modes are indistinguishable and an alarm can be raised for attack detection. %\vspace{-0.1cm}
\end{proof}

\subsection{Attack Identification}

On the other hand, even when a combined switching and false data injection attack is detectable, it may not be identifiable. In order to identify an attack strategy and action, the mean consistency property of our estimation algorithm is a sufficient condition, which follows directly by Definition \ref{def:identification}.

%The following corollary holds directly by Definition \ref{def:identification}.

\begin{theorem}[Attack Identification] Suppose mean consistency, i.e., \cite[Theorem 8]{Yong.Zhu.ea.SICON16} %Theorem \ref{thm:mean_cons} 
holds (and hence \cite[Corollary 13]{Yong.Zhu.ea.SICON16} also holds). Then, the switching and data injection attack strategy can be identified using the resilient state estimation algorithm in Algorithm \ref{algorithm2}.
\end{theorem}

On the other hand, if the estimator is not mean consistent but the true mode is in the set of models, then %by \cite[Theorem 8]{Yong.Zhu.ea.SICON16}, %Theorem \ref{thm:mean_cons}, 
there must exist some models with generalized innovations that have identical probability distributions as the generalized innovation of the true model (since their KL-divergences are identically zero), and that are hence Gaussian white sequences \cite{Yong.Zhu.ea.SICON16}. 
%%Note that it has also been independently observed that attack stealthiness for SISO stochastic systems is directly related to KL-divergence \cite{Bai.2015}.
In other words, in order to remain unidentifiable for some mode $q \in \mathcal{Q}$, an attacker seeks to choose another `true' mode  $* \in \mathcal{Q}, * \neq q$ and the attack signal $d_k$ such that the distributions of their generalized innovations $\nu_k^*$ and $\nu_k^{q|*}$ are identical, i.e., Gaussian white sequences with $\mathbb{E}[\nu_k^{q|*}]=\mathbb{E}[\nu_k^*]=0$ and $\mathbb{E}[\nu_k^{q|*} \nu_k^{{q|*} \top}]=\mathbb{E}[\nu_k^*\nu_k^{* \top}]=S_k^* \triangleq \tilde{\Gamma}_k^* \tilde{R}_{2,k}^{*,\star} \tilde{\Gamma}_k^{* \top} $ for all $k$.
Using this observation, we now investigate some conditions under which attackers can be unidentifiable, as well as some other conditions under which the defenders/system operators can guarantee that the attacks are identifiable.

\subsubsection{A sufficient condition for unidentifiable attacks}
Given that mean consistency guarantees that an attack is identifiable, the goal of an attacker would be to ensure that mean consistency does not hold by a strategic choice of data injection signals $d_k$ and the true mode $* \in \mathcal{Q}$ in order to mislead the system operator/defender into believing that the mode of operation is $q \in \mathcal{Q}, q \neq *$. The following is a sufficient condition for an attacker to synthesize an unidentifiable switching and data injection attack.

%If such attack cases can be carried out when the system operator uses the resilient estimator in Algorithm \ref{algorithm2}, then we refer to these attacks as \yong{\emph{unidentifiable}}, formally defined as follows:

%\begin{definition}[Attack Identification] \label{def:identification} An attack strategy is identified if the true mode $* \in \mathcal{Q}$ is uniquely determined on average, which reveals the chosen \emph{mode attack} and \emph{signal attack location}, and asymptotically unbiased estimates of attack signals $d_k$ can be obtained, i.e., the \emph{signal magnitude attack} is reliably estimated. 
%\end{definition}
%
%The following corollary holds directly by the above definition.
%
%\begin{corollary}[Attack Identification] Suppose mean consistency, i.e., Theorem \ref{thm:mean_cons} holds (and hence Corollary \ref{cor:opt} also holds). Then, the switching and data injection attack strategy can be identified using the resilient state estimation algorithm in Algorithm \ref{algorithm2}.
%\end{corollary}

%\begin{definition}[Undetectable Switching and Data Injection Attack] \label{def:undetectable} Under the normal mode of operation, $q \in \mathcal{Q}$, the combination of a mode attack and a data injection attack is \emph{undetectable}, if the attacked mode and signal attack locations corresponding to $* \in \mathcal{Q}, * \neq q$ and signal magnitude attacks $d_k$ are such that  $D(f^\ast_\ell \| f^q_\ell) = 0$ (i.e., Algorithm \ref{algorithm2} is not mean consistent according to Theorem \ref{thm:mean_cons}).
%\end{definition}

\begin{theorem}[Unidentifiable Attack] \label{thm:attacker}
Suppose $\tilde{\Gamma}^q_k T_{2,k}^q H^*_k$ has linearly independent rows and there exists $* \neq q \in \mathcal{Q}$ such that
\begin{align} \label{eq:stochasticD}
\mathcal{D}^s_k \triangleq (\tilde{\Gamma}^q_k T_{2,k}^q H^*_k)^\dagger (S^*_k-\tilde{\Gamma}^q_k T_{2,k}^q (\mathbb{E}[\mu^{q|*}_k \mu^{q|* \top}_k]+R_k)(\tilde{\Gamma}^q_k T_{2,k}^q)^\top))(\tilde{\Gamma}^q_k T_{2,k}^q H^*_k)^{\dagger \top} 
\end{align}
is positive definite ($\succeq 0$) for all $k$. Moreover, we assume that $\mu_0^*=\mu_0^q$. Then, the attack is unidentifiable if the attacker chooses this mode $* \neq q$ as well as the attack signal $d_k$ as a Gaussian sequence
\begin{align} \label{eq:d_attack}
d_k \sim \mathcal{N} (d_k^d, \mathcal{D}^s_k), \quad \forall k
\end{align}
with $\mathcal{D}^s_k$ defined in \eqref{eq:stochasticD} and $d_k^d$ is given by
\begin{align} \label{eq:d_det}
\begin{array}{rl}
d^d_k \triangleq \mathbb{E}[d_k] &= -(\tilde{\Gamma}^q_k T_{2,k}^q H^*_k)^\dagger \tilde{\Gamma}^q_k T_{2,k}^q (C_k^* \mathbb{E}[x_k] - C_k^q \hat{x}_{k|k}^{\star,q}+(D_k^*-D_k^q)\mathbb{E}[u_k])\\
&= -(\tilde{\Gamma}^q_k T_{2,k}^q H^*_k)^\dagger \tilde{\Gamma}^q_k T_{2,k}^q (C_k^* \hat{x}_{k|k}^* - C_k^q \hat{x}_{k|k}^{\star,q}+(D_k^*-D_k^q)\mathbb{E}[u_k]), \quad \forall k.
\end{array}
\end{align}
\end{theorem}
\begin{proof}
First, we compute what the generalized innovation for $q \in \mathcal{Q}$ would be when the attacker chooses $*$ as the true mode:
\begin{align} \label{eq:nu_qstar}
\begin{array}{rl}
\nu^{q|*}_k&= \tilde{\Gamma}^q_k T_{2,k}^q ( y_k - C_k^q \hat{x}_{k|k}^{\star,q} -D_k^q u_k)
=  \tilde{\Gamma}^q_k T_{2,k}^q(\mu^{q|*}_k+H^*_k (d_k-\mathbb{E}[d_k])+v_k),
\end{array}
\end{align}
where $\mu^{q|*}_k\triangleq C_k^* x_k - C_k^q \hat{x}_{k|k}^{\star,q}+(D_k^*-D_k^q)u_k + H^*_k d^d_k$, whose first and second moments, $\mathbb{E}[\mu^{q|*}_k]$ and $\mathbb{E}[\mu^{q|*}_k \mu^{q|* \top}_k]$, are assumed to be known to the attacker, while $T_{2,k}^q$ is the transformation matrix for mode $q$ as described in detail in \cite{Yong.Zhu.ea.Automatica16}.

Substituting \eqref{eq:d_attack} into \eqref{eq:nu_qstar} and computing its first and second moments using \eqref{eq:stochasticD} and \eqref{eq:d_det}, we obtain
\begin{align*}
\mathbb{E}[\nu^{q|*}_k] &=  \tilde{\Gamma}^q_k T_{2,k}^q\mathbb{E}[\mu^{q|*}_k]=0,\\
\mathbb{E}[\nu^{q|*}_k \nu^{q|* \top}_k] &= \tilde{\Gamma}^q_k T_{2,k}^q ( \mathbb{E}[\mu^{q|*}_k \mu^{q|* \top}_k] +H^*_k\mathcal{D}^s_k H^{*\top}_k + R_k) (\tilde{\Gamma}^q_k T_{2,k}^q)^\top = S^*_k.
\end{align*}
%Since \cite[Theorem 8]{Yong.Zhu.ea.SICON16} %Theorem \ref{thm:mean_cons} 
%does not hold, there is no guarantee that our resilient estimator in Algorithm \ref{algorithm2} is mean consistent. Moreover, 
Since we assumed that $\mu_0^*=\mu_0^q$, %from \eqref{eq:geometric_mean}, 
we observe that the ratio of the geometric means of model probabilities given in \cite[Lemma 14]{Yong.Zhu.ea.SICON16} equals 1. In other words, the attacked system cannot be distinguished from one under normal operation, i.e., the attack is unidentifiable by Definition \ref{def:identification}.
\end{proof}

From the above theorem, we observe that the unidentifiable attack strategy relies on two factors. First, the system has vulnerabilities that can be exploited, if the sufficient conditions of the theorem are allowed to hold. Thus, as a system designer, these conditions serve as a guide for securing the system. Secondly, the attacker needs computational capability and system knowledge that are comparable to that of the system operator/defender.

\subsubsection{A sufficient condition for resilient state estimation}

From the perspective of the system designer/operator/defender, the main objective of resilient state estimation is to obtain unbiased state estimates in order to preserve the integrity and functionality of the system despite attacks. Attack identification is a secondary goal and is only optional. 

First and foremost, the system vulnerabilities need to be eliminated. Thus, the system needs to be strongly detectable for all modes $q \in \mathcal{Q}$, as discussed in Section \ref{sec:sysDes}. Next, since Theorem \ref{thm:attacker} presents yet another system vulnerability, a sufficient condition is needed such that this theorem does not hold.

\begin{lemma}
Theorem \ref{thm:attacker} does not hold if $H^q_k = H_k$ for all $q \in \mathcal{Q}$.
\end{lemma}
\begin{proof}
Since $H^q_k = H_k$ for all $q \in \mathcal{Q}$, $T_{2,k}^q=T_{2,k}^*$ and thus, $T_{2,k}^q H_k^*=T_{2,k}^*H_k^*=0$ and $\tilde{\Gamma}^q_k T_{2,k}^q H^*_k=0$. Hence, $\tilde{\Gamma}^q_k T_{2,k}^q H^*_k=0$ cannot have linearly independent rows, as assumed in Theorem \ref{thm:attacker}.
\end{proof}

In addition to requiring $H^q_k = H_k$ for all $q \in \mathcal{Q}$, without loss of generality and for simplicity, we also assume that $D_k^q=D_k^*$ for all $q \in \mathcal{Q}$. Since Theorem \ref{thm:attacker} presents only  sufficient conditions, a strategic attacker may somehow still be able to make the  distributions of the generalized innovations $\nu_k^*$ and $\nu_k^{q|*}$ identical. Thus, even in this case where $H^q_k = H_k$  and $D^q_k = D_k$ for all $q \in \mathcal{Q}$, it is interesting to investigate further sufficient conditions for the system defender to ensure that the main objective of resilient state estimation is still achieved.

\begin{theorem}[Resilience Guarantee] \label{thm:guarantee}
Suppose $H^q_k = H_k$  and $D^q_k = D_k$ for all $q \in \mathcal{Q}$. Moreover, for all $q,q' \in \mathcal{Q}$, if there exists $T$ such that for all $k \geq T$ and the following holds
\begin{enumerate}[(i)]
\item $\rm{rank} \begin{bmatrix} \tilde{\Gamma}^{q}_k T_{2,k}^q C^{q'}_{k} & \tilde{\Gamma}^q_k T_{2,k}^q C_{k}^q \end{bmatrix}=2 n$, if $C_k^{q} \neq C_k^{q'}$,
\item $\rm{rank}(\tilde{\Gamma}^{q}_k T_{2,k}^qC^{q'}_{k}) =\rm{rank}(\tilde{\Gamma}^q_k T_{2,k}^qC^{q}_{k} )=n$, if $C_k^{q} = C_k^{q'}$,
\end{enumerate}
then the state estimates obtained using Algorithm \ref{algorithm2} are guaranteed to be resilient (i.e., asymptotically unbiased).
\end{theorem}
\begin{proof}
These sufficient conditions are derived by making sure that $\mathbb{E}[\nu_k^{q|*}] \neq \mathbb{E}[\nu_k^*]=0$ such that \cite[Theorem 8]{Yong.Zhu.ea.SICON16} %Theorem \ref{thm:mean_cons} 
does not hold. First, by the assumptions of this theorem,  \eqref{eq:nu_qstar} simplifies to 
\begin{align} %\label{eq:nu_qstar}
\begin{array}{rl}
\nu^{q|*}_k&= \tilde{\Gamma}^q_k T_{2,k}^q ( C_k^* x_k - C_k^q \hat{x}_{k|k}^{\star,q} +v_k).
\end{array}
\end{align}

In Case (i), i.e., when $C_k^{q} \neq C_k^{*}$, we have $\mathbb{E}[\nu^{q|*}_k]=\begin{bmatrix} \tilde{\Gamma}^{q}_k T_{2,k}^q C^{*}_{k} & \tilde{\Gamma}^q_k T_{2,k}^q C_{k}^q \end{bmatrix} \begin{bmatrix} \mathbb{E}[{x}_{k}] \\ -\mathbb{E}[\hat{x}_{k|k}^{\star,q}] \end{bmatrix}$. Hence, since the rank condition holds, $\mathbb{E}[\nu^{q|*}_k] \neq 0$ unless $\begin{bmatrix} \mathbb{E}[{x}_{k}] \\ -\mathbb{E}[\hat{x}_{k|k}^{\star,q}] \end{bmatrix}=0$, in which case we have an unbiased (and thus resilient) estimate $\mathbb{E}[\hat{x}_{k|k}^{\star,q}]=\mathbb{E}[{x}_{k}]=0$. 

In Case (ii), i.e., when $C_k^{q} = C_k^{*}$, we have $\mathbb{E}[\nu^{q|*}_k]= \tilde{\Gamma}^{q}_k T_{2,k}^q C^{q}_{k}  \mathbb{E}[{x}_{k}-\hat{x}_{k|k}^{\star,q}]$. By the rank assumption, $\mathbb{E}[\nu^{q|*}_k] \neq 0$ unless $\mathbb{E}[{x}_{k}-\hat{x}_{k|k}^{\star,q}]=0$, in which case we again have an unbiased (and thus resilient) estimate $\mathbb{E}[\hat{x}_{k|k}^{\star,q}]=\mathbb{E}[{x}_{k}]$.
\end{proof}

%\subsection{Attack Identification}
%In this section, we discuss some conditions under which the attacker's strategy and actions can be uniquely identified (cf. definition below).
%
%\begin{definition}[Attack Identification] \label{def:identification} An attack strategy is identified if the true mode $* \in \mathcal{Q}$ is uniquely determined on average, which reveals the chosen \emph{mode attack} and \emph{signal attack location}, and asymptotically unbiased estimates of attack signals $d_k$ can be obtained, i.e., the \emph{signal magnitude attack} is reliably estimated. 
%\end{definition}
%
%The following corollary holds directly by the above definition.
%
%\begin{corollary}[Attack Identification] Suppose mean consistency, i.e., Theorem \ref{thm:mean_cons} holds (and hence Corollary \ref{cor:opt} also holds). Then, the switching and data injection attack strategy can be identified using the resilient state estimation algorithm in Algorithm \ref{algorithm2}.
%\end{corollary}

Note that Theorem \ref{thm:guarantee} only guarantees that the state estimates are unbiased, but the true mode cannot be uniquely distinguished and the attack signal cannot be estimated. Hence, the conditions in Theorem \ref{thm:guarantee} are not sufficient for attack identification.

\section{Attack Mitigation}\label{sec:mitigation}

We now turn to the next step beyond attack detection and identification, and investigate how we can mitigate the effects of attacks (Problem 1.\ref{prob4}). Specifically, we study the problem of rejecting/canceling data injection attacks assuming that the attack mode can be identified (thus, the superscript $q$ is omitted throughout this section), while using the resilient state estimates for feedback stabilization, in the sense of guaranteeing the boundedness of the expected states for bounded attack signals.
To this end, we consider a linear state feedback controller with attack/disturbance rejection terms, where the true state and unknown input are replaced by their estimated values: % as follows: 
\begin{align} \label{eq:DTcontrol}
u_k=-K^c_k \hat{x}_{k|k} - J_{1,k} \hat{d}_{1,k} - J_{2,k} \hat{d}_{2,k-1},
\end{align} \noindent
where $K^c_k$ is the state feedback gain, while $J_{1,k}$ and $J_{2,k}$ are the attack/disturbance rejection gains. Note that we have used a delayed estimate of $d_{2,k-1}$ given in \eqref{eq:d}, which is the only estimate we can obtain in light of \cite[Equation (6)]{Yong.Zhu.ea.Automatica16}. 

\begin{theorem}[Attack-Mitigating and Stabilizing Controller] \label{thm:control} 
Suppose the system is controllable in the true mode $q \in \mathcal{Q}$ (known or detected), and the expected values of attack signals $\mathbb{E}[d_{2,k}]$ and their rates of variation $\|\mathbb{E}[d_{2,k}-d_{2,k-1}]\|$  are bounded for all $k$, i.e., $\|\mathbb{E}[d_{2,k}]\| \leq d^B_2$ and  $\|\mathbb{E}[d_{2,k}-d_{2,k-1}]\| \leq \delta d^B_2$ with $\min\{d^B_2,\delta d^B_2\} < \infty$. Then, a feedback controller that  mitigates the effects of data injection attacks and guarantees the boundedness of the expected system states is given by
\begin{align} \label{eq:controlexp}
u_k=-K^c_k \hat{x}_{k|k} - \begin{bmatrix} J_{1,k} & J_{2,k} \end{bmatrix} \tilde{J}_k^{-1} \begin{bmatrix} M_{1,k} (z_{1,k}-C_{1,k} \hat{x}_{k|k}) \\ M_{2,k} (z_{2,k} - C_{2,k} \hat{x}_{k|k-1}) \end{bmatrix},
\end{align}
where $K^c_k$ is any state feedback gain such that $A_k -B_k K^c_k$ is stable and the attack/disturbance rejection gain $J_{1,k}$ is chosen to minimize $ \gamma_1 \triangleq \| G_{1,k}-B_k J_{1,k}\|_2$, which can be solved with a semidefinite program\footnote{Semidefinite programs are convex optimization problems for which software packages, e.g. CVX \cite{cvx,Grant.08}, 
are available.} (with $i=1$) as follows:
\begin{align} \label{eq:sdp}
\begin{array}{ll}
{\rm minimize\quad  } &\gamma_i \\
{\rm subject \ to \ }  &\begin{bmatrix} \gamma_i I & G_{i,k}-B_k J_{i,k} \\  (G_{i,k}-B_k J_{i,k})^\top & \gamma_i I \end{bmatrix} \succeq 0, 
\end{array}
\end{align} 
while $J_{2,k}$ is chosen as 0 if $\delta d^B_2 >  d^B_2$, and otherwise, to  minimize $ \gamma_2 \triangleq \| G_{2,k}-B_k J_{2,k}\|_2$ by solving   the semidefinite program \eqref{eq:sdp} with $i=2$. It is assumed that $\tilde{J}_k \triangleq \begin{bmatrix} I-M_{1,k} D_{1,k} J_{1,k} & -M_{1,k} D_{1,k} J_{2,k} \\ -M_{2,k} D_{2,k} J_{1,k} & I-M_{2,k} D_{2,k} J_{2,k} \end{bmatrix}$ is invertible.
\end{theorem}

To prove Theorem \ref{thm:control}, we first show that there exists a separation principle for linear discrete-time stochastic systems with unknown inputs, i.e., when the true mode is known, which allows us to choose the state feedback gain $K^c_k$ and attack/disturbance rejection gains $J_{1,k}$ and $J_{2,k}$ independently. 
\begin{lemma}(Separation Principle) \label{lem:DTseparation}
The state feedback controller gain $K^c_k$ in \eqref{eq:DTcontrol} can be designed independently of the state and input estimator gains $L_k$, $M_{1,k}$ and $M_{2,k}$ in Algorithm \ref{algorithm1}, as well as the disturbance rejection gains $J_{1,k}$ and $J_{2,k}$.
\end{lemma}
\begin{proof}
Using the control law \eqref{eq:DTcontrol} and the filter equations in \eqref{eq:d}, \eqref{eq:xstar} and \eqref{eq:stateEst}, it can be verified that the system states and estimator error dynamics are given by 
\begin{align} 
\nonumber \begin{bmatrix} {x}_{k+1} \\ {\tilde{x}}_{k+1|k+1} \end{bmatrix} &= \begin{bmatrix} A_k-B_k K^c_k & \begin{array}{c} B_k (K_k^c \hspace{-0.05cm}-\hspace{-0.05cm} J_{1,k} M_{1,k} C_{1,k} \hspace{-0.05cm} \\ - J_{2,k} M_{2,k} {C}_{2,k+1}(A_k \hspace{-0.05cm} + \hspace{-0.05cm} G_{1,k} M_{1,k} C_{1,k}))\end{array} \\ 0 & (I- \tilde{L}_{k+1} C_{2,k}) \overline{A}_k \end{bmatrix} \begin{bmatrix} {x}_k \\ {\tilde{x}_{k|k}} \end{bmatrix}\\
 &  + \begin{bmatrix} G_{1,k}-B_k J_{1,k} & G_{2,k}-B_k J_{2,k} \\ 0 & 0 \end{bmatrix} \begin{bmatrix} d_{1,k} \\ d_{2,k} \end{bmatrix} + \begin{bmatrix} B_k J_{2,k} \\ 0 \end{bmatrix} ({d}_{2,k}-\hat{d}_{2,k-1}) \label{eq:closedloop} \\
\nonumber &   + \yong{\begin{bmatrix}  \begin{array}{l} I -B_k J_{2,k} \\ \qquad  M_{2,k+1} C_{2,k+1}\end{array}    & \begin{array}{r} - B_k J_{1,k} M_{1,k} + B_k J_{2,k} \ \\ M_{2,k+1} {C}_{2,k+1} G_{1,k} M_{1,k}  \end{array}  &  -B_k J_{2,k} M_{2,k+1} \\ \begin{array}{l} (I-\tilde{L}_{k+1} C_{2,k+1})\\ (I-G_{2,k}M_{2,k+1}\\ \qquad C_{2,k+1}) \end{array} & \begin{array}{l}-(I-\tilde{L}_{k+1} C_{2,k+1})\\ (I-G_{2,k}M_{2,k+1} C_{2,k+1})\\ G_{1,k}M_{1,k}
 \end{array} & \begin{array}{l}-(I-\tilde{L}_{k+1} C_{2,k+1})\\
\ \ G_{2,k}M_{2,k+1}- \tilde{L}_{k+1} \end{array} \end{bmatrix}} \mathbf{w}_k,
\end{align}
where \yong{$\mathbf{w}_k\triangleq \begin{bmatrix} w_k^\top  & v_{1,k}^\top & v_{2,k+1}^\top \end{bmatrix}^\top$} has zero mean and $\overline{A}_k \triangleq (I-G_{2,k-1} M_{2,k} C_{2,k})(A_k-G_{1,k} M_{1,k} C_{1,k})$. %and \yong{$\overline{w}_k \triangleq - (I-G_{2,k}M_{2,k+1}C_{2,k+1})(G_{1,k} M_{1,k}v_{1,k}-w_k)-G_{2,k}M_{2,k+1}v_{2,k+1}$.} 
Since the state matrix has a block diagonal structure, the eigenvalues of the controller and estimator are independent of each other.
\end{proof}

Armed with the above lemma, we now show how the state feedback and attack rejection gains can be independently chosen. 

\begin{proof}[of Theorem \ref{thm:control}]
By Lemma \ref{lem:DTseparation}, the state feedback gain, $K^c_k$, can be independently designed  with no effect on the stability of the resilient state estimator and independent of the choice of the disturbance rejection gains $J_{1,k}$ and $J_{2,k}$. In other words, $K^c_k$ can be chosen as any state feedback gain (e.g., with Linear Quadratic Regulator (LQR) or pole placement) such that $A_k - B_k K^c_k$ is stable, thus the expected system states is bounded since the expected values of the attack signals and their variation rates are bounded by assumption.

On the other hand, $J_1$ and $J_2$ are chosen such that the effect of injected attack signal $d_k$ on the closed loop system is minimized/reduced. Since $d_{1,k}$ affects the closed loop dynamics through the matrix $G_{1,k}-B_k J_{1,k}$, we choose $J_{1,k}$ such that the induced 2-norm of $G_{1,k}-B_k J_{1,k}$ is minimized, which can be obtained by the semidefinite program \eqref{eq:sdp} with $i=1$. Similarly, $J_{2,k}$ can be chosen to minimize the induced 2-norm of $G_{2,k}-B_k J_{2,k}$ using the semidefinite program \eqref{eq:sdp} with $i=2$.
However, %unlike the separation principle for linear continuous-time systems in \cite{Yong.Zhu.ea.TAC16}, 
we have an additional ${d}_{2,k}-\hat{d}_{2,k-1}$ in the closed loop dynamics \eqref{eq:closedloop}, which disappears if $J_{2,k}=0$. Thus, if we assume that $\mathbb{E}[d_{2,k}]$ is bounded for all $k$, i.e., $\|\mathbb{E}[d_{2,k}]\| \leq d^B_2$, and its rate of variation is  bounded, i.e., $\|\mathbb{E}[d_{2,k}-d_{2,k-1}]\| \leq \delta d^B_2$ is bounded for all $k$, one would choose $J_{2,k}$ as 0 or use \eqref{eq:sdp} with $i=2$, depending on the lower of the two bounds, $d^B_2$ or $\delta d^B_2$.
 
 In addition, $J_{1,k}$ and $J_{2,k}$ must also be chosen so that $u$, $\hat{d}_{1,k}$ and $\hat{d}_{2,k-1}$ can be uniquely determined since $\hat{d}_{1,k}$ and $\hat{d}_{2,k-1}$ become implicit equations. Thus, the choices of $J_{1,k}$ and $J_{2,k}$ must also be such that $\tilde{J}_k$ is invertible. The explicit expressions for $\hat{d}_{1,k}$ and $\hat{d}_{2,k-1}$ (to be substituted into Algorithm \ref{algorithm1}) is then 
\begin{align} \label{eq:dexp}
\begin{bmatrix} \hat{d}_{1,k}\\ \hat{d}_{2,k-1} \end{bmatrix} = \tilde{J}_k^{-1} \begin{bmatrix} M_{1,k} (z_{1,k}-C_{1,k} \hat{x}_{k|k}) \\ M_{2,k} (z_{2,k} - C_{2,k} \hat{x}_{k|k-1}) \end{bmatrix}.
\end{align}

Substituting \eqref{eq:dexp} back into \eqref{eq:DTcontrol}, we obtain the feedback controller in \eqref{eq:controlexp}.
\end{proof}

Note that if the system in \eqref{eq:system} for each $q_k \in \mathcal{Q}$ fulfills a \emph{matching condition} for $d_{1,k}$ \footnote{The matching condition assumption is common for disturbance rejection in the sliding mode and adaptive control literature.}, i.e., $\exists J_{1,k}$ such that $B_k J_{1,k} d_{1,k}=G_{1,k} d_{1,k}$, the above minimization procedure will exactly cancel out the attack signal $d_{1,k}$. 

\section{Simulation Examples} \label{sec:examples}

\subsection{Benchmark System (Signal Magnitude \& Location Attacks)}\label{eg2}

\begin{figure}[!tp]
\begin{center} 
\includegraphics[scale=0.575,trim=1mm 55mm 0mm 0mm,clip]{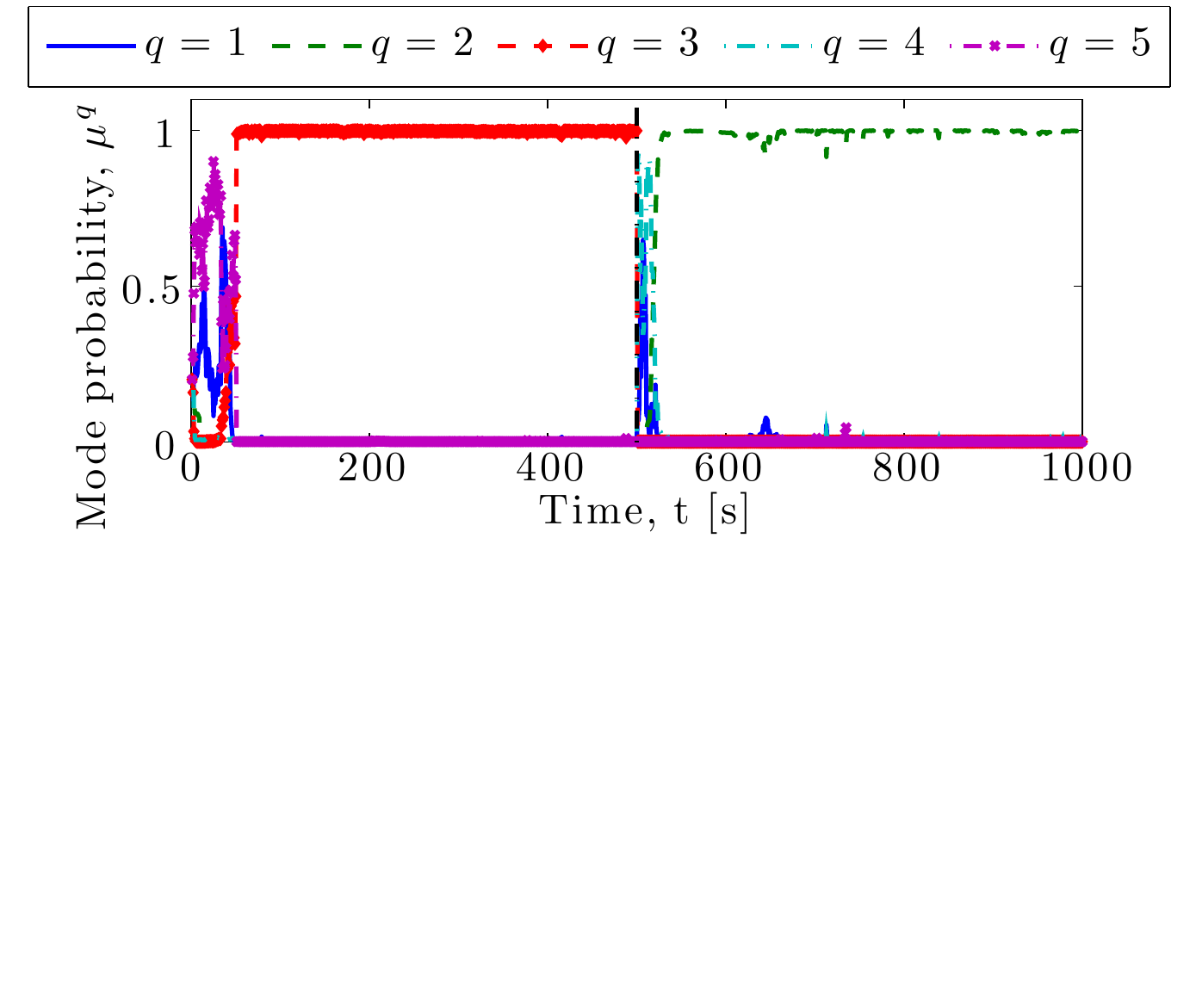}
\caption{Mode probabilities for Example \ref{eg2}.\label{fig:modeA}}
\includegraphics[scale=0.6,trim=23mm 6mm 15mm 8mm,clip]{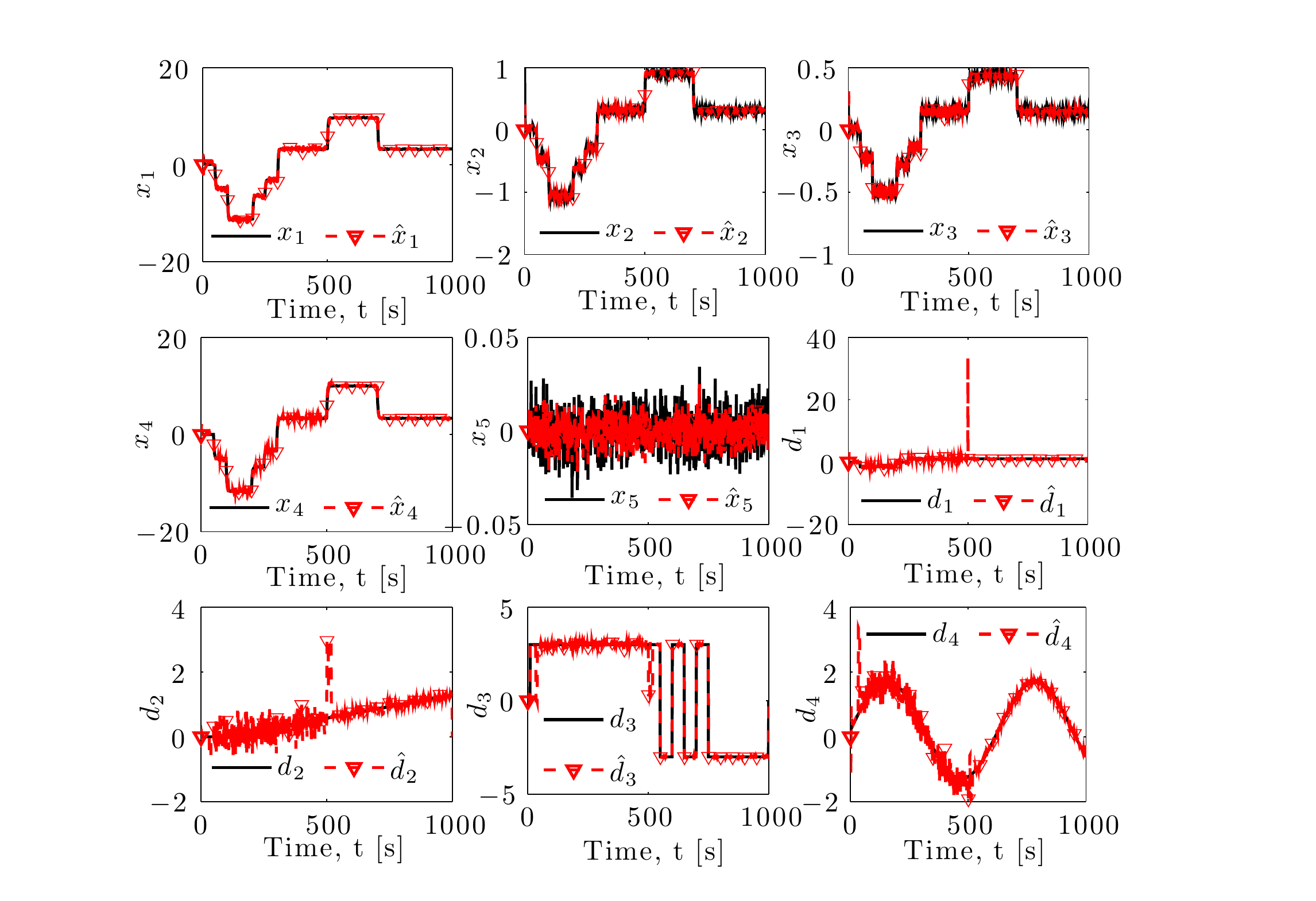}
\caption{State and attack magnitude estimates in Example \ref{eg2}.\label{fig:statesA}} \vspace{-0.3cm}
\end{center}
\end{figure}

In this example, we consider the resilient state estimation problem for a system (modified from \cite{Yong.Zhu.ea.Automatica16}) that has been used as a benchmark for many state and input filters, with only one mode of operation ($t_m=1$) and with possible attacks on the actuator and 4 of the 5 sensors ($t_a=1$, $t_s=4$): 
\begin{align*}
A \hspace{-0.05cm}&=\hspace{-0.15cm} \begin{bmatrix} 0.5 & 2 & 0 & 0 & 0\\ 0 & 0.2 & 1 & 0 &1 \\ 0 & 0 & 0.3 & 0 & 1 \\ 0 & 0 & 0 & 0.7 & 1 \\ 0 & 0 & 0 & 0 & 0.1\end{bmatrix}\hspace{-0.1cm}; \quad  
  B\hspace{-0.05cm}=\hspace{-0.05cm}G\hspace{-0.05cm}=\hspace{-0.05cm} \begin{bmatrix} 1 \\ 0.1 \\ 0.1  \\1 \\0 \end{bmatrix}\hspace{-0.1cm}; \quad
C \hspace{-0.05cm}=\hspace{-0.05cm} \begin{bmatrix}1 & 0 & 0& 0 &0\\0 & 1& -0.1& 0& 0 \\ 0 & 0 & 1& -0.5 & 0.2\\ 0 & 0 & 0 & 1 & 0\\ 0 & 0.25 & 0 & 0 & 1\end{bmatrix}\hspace{-0.1cm}; \\ 
H\hspace{-0.05cm}&=\hspace{-0.15cm}\begin{bmatrix} 1 & 0 & 0 & 0 \\ 0 & 1 & 0 & 0 \\ 0 & 0 & 1 & 0 \\ 0 & 0 & 0 & 1 \\ 0 & 0 & 0 & 0 \end{bmatrix}\hspace{-0.1cm}; \quad Q \hspace{-0.05cm}=\hspace{-0.05cm} 10^{-4}\hspace{-0.1cm} \begin{bmatrix} 1 & 0 & 0 & 0 & 0  \\ 0 & 1 & 0.5& 0 & 0\\ 0& 0.5 & 1 & 0 & 0 \\ 0 & 0 & 0 & 1 & 0 \\ 0 & 0 & 0 & 0 & 1 \end{bmatrix}\hspace{-0.1cm}; \quad 
R \hspace{-0.05cm}=\hspace{-0.05cm} 10^{-4}\hspace{-0.1cm}\begin{bmatrix}  1 & 0 & 0 & 0.5 & 0\\0 & 1 & 0 & 0 & 0.3\\0 & 0 & 1 & 0 & 0\\ 0.5 & 0 & 0 & 1 & 0 \\ 0 & 0.3 & 0 & 0 & 1\end{bmatrix}\hspace{-0.1cm}.
\end{align*} 
\indent The known input $u_k$ is $2$ for $100 \leq k \leq 300$, $-2$ for $500 \leq k \leq 700$ and $0$ otherwise,
whereas the unknown inputs are as depicted in Figure \ref{fig:statesA}.
We also assume that there are at most $p=4$ attacks with no constraints on $n_a$ and $n_s$; as a result, we have to consider $\mathfrak{N}=1 \cdot \binom{5}{4}=5$ models.

Due to space limitation, we only provide simulation results for the case when the signal attack locations are switched from $q=3$ (attack on actuator and sensors 1,3,4)  to $q=2$ (attack on actuator and sensors 1,2,4) at time $t=500s$. From Figure \ref{fig:modeA}, we observe that the mode probabilities converge to their true values. Figure \ref{fig:statesA} shows the estimates of states as well as the unknown attack signal magnitudes. The state estimates, which are our main concern, are seen to be good even before the mode probabilities converge, while the unknown attack signals are also reasonably well estimated, with the exception of little jumps in its estimates during the switch in attack locations at $t=500s$. Similar results (not shown) are obtained for all other attack modes, $q=1$ (attack on actuator and sensors 1,2,3), $q=4$ (attack on actuator and sensors 2,3,4) and $q=5$ (attack on sensors 1,2,3,4).

\yong{
\subsection{IEEE 68-Bus Test System (Mode \& Signal Magnitude Attacks)}\label{eg1}

\begin{figure}[!bp]
\begin{center}
\includegraphics[scale=0.395,trim=1mm 1mm 1mm 0mm,clip]{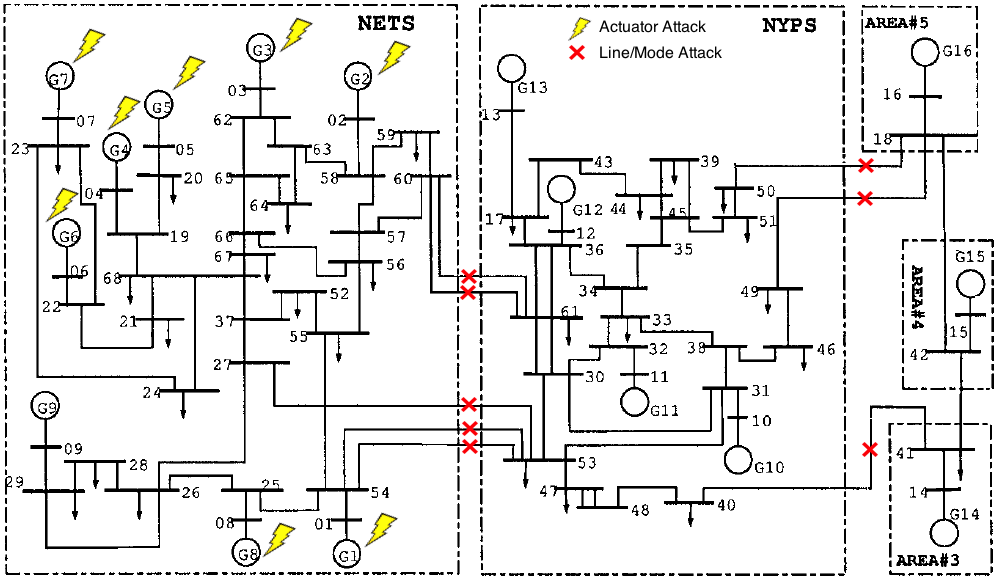}
\caption{\yong{IEEE 68-bus test system with locations of potential actuator signal and mode/transmission line attacks (adapted from \protect\cite{pal2006robust})}.\label{fig:IEEE68bus}}  \vspace{-0.3cm}
\end{center}
\end{figure}

Next, we apply our approach to the IEEE 68-bus test system shown in Figure \ref{fig:IEEE68bus} 
to empirically illustrate that the proposed algorithm can scale to large systems. 

A power network is generally represented by undirected graph
$(\mathcal{V}, \mathcal{E})$ with the set of buses $\mathcal{V}\triangleq\{1, · · · , N\}$ and the set
of transmission/tie lines $\mathcal{E} \subseteq \mathcal{V} \times \mathcal{V}$.  Each
bus is either a generator bus $i \in \mathcal{G}$, or a load bus $i \in \mathcal{L}$. The set of neighboring
buses of $i \in \mathcal{V}$ is denoted as $\mathcal{S}_i \triangleq \{j \in \mathcal{V} \setminus \{i\}|(i,j) \in \mathcal{E}\}$.
For the IEEE 68-bus test system, there are 16 generator buses and 52
load buses (i.e., $|\mathcal{G}|=16$, $|\mathcal{L}|=52$ and $|\mathcal{V}|=68$). Each bus, $i \in \mathcal{V}$, is described by (as in \cite[Chap. 10]{Wood.2013}):
\begin{align}\label{eq:busDyn}
\begin{array}{rl}
\dot{\theta}_i(t)&=\omega_i(t),\\
\dot{\omega}_i(t)&=-\frac{1}{m_i}[D_i \omega_i(t) +\sum_{j \in \mathcal{S}_i} P^{ij}_{tie} (t)-(P_{M_i}(t)+d_{a,i}(t))+P_{L_i}(t) +w_i (t)],
\end{array}
\end{align}
with the phase angle $\theta_i(t)$ and angular frequency  $\omega_i(t)$ as
system states (thus, the state dimension is $n=136$) and an actuator attack signal $d_{a,i}(t)$. The power flow between neighboring buses $(i,j) \in \mathcal{E}$ is given by $P^{ij}_{tie}(t)=-P^{ji}_{tie}(t)=t_{ij} (\theta_i(t)-\theta_j (t))$, while the mechanical power and power demand are denoted as $P_{M_i} (t)$ and $P_{L_i} (t)$, respectively.
The mechanical power $P_{M_i} (t)$ is the control input for the generator bus $i \in \mathcal{G}$
and is zero at load bus $i \in \mathcal{L}$. On the other hand, since power demand $P_{L_i} (t)$
 can be obtained using load forecasting methods (e.g., \cite{alfares2002electric}), it is assumed to be a known input to the system. It is assumed that the noise $w_i(t)$ is a zero-mean Gaussian signal with covariance matrix $Q_i(t) = 0.01$ and the system parameters
are adopted from \cite[page 598]{kundur1994power}: $D_i = 1$, 
$t_{ij} = 1.5$ for all $i \in \mathcal{V}, j \in \mathcal{S}_i$ and $t_{ij}=0$ otherwise. Angular momentums are $m_i = 10$
for $i \in \mathcal{G}$ and a larger value $m_i = 100$ for load buses $i \in \mathcal{L}$.

 %In order to apply our proposed approach to this example, with a sampling time of $\Delta t= 0.01$s. 

The measurements are sampled at discrete times (with sampling time $\Delta t= 0.01$s):% and we consider the following output/measurement model:
\begin{align}
y_{i,k} = \begin{bmatrix}P_{elec,i,k}& \theta_{i,k}& \omega_{i,k}\end{bmatrix}^\top + v_{i,k}, % =\begin{bmatrix}D_i \omega_{i,k} + P_{L_i,k}\\ \theta_{i,k}\\ \omega_{i,k}\end{bmatrix} + v_{i,k},
\end{align}
where $P_{elec,i,k}=D_i \omega_{i,k} + P_{L_i,k}$ is the electrical power output and  $v_{i,k}$ is a zero-mean Gaussian noise signal with covariance matrix $R_i(t) = 0.01^4 I_{3}$. The continuous system dynamics \eqref{eq:busDyn} is also discretized with a sampling time of $\Delta t= 0.01$s so that it is compatible with the measurement model. Moreover, in this example, we choose stabilizing control inputs $P_{M_i,k}$ and $P_{L_i,k}$  to regulate the phase angles to $\theta_i=10$ rad with system eigenvalues at $-0.05$ using standard linear control design tools, which is combined with the attack-mitigating controller described in Theorem \ref{thm:control}.

The attacker could launch actuator attacks and mode/transmission line attacks as shown in Figure \ref{fig:IEEE68bus}. For this case study, we consider 8 potential attacks modes ($|\mathcal{Q}|=8$):
\begin{description}
\item[Mode $q=1$] Lines \{27,53\},\{53,54\},\{60,61\} \& actuator $G1$.
\item[Mode $q=2$] Lines \{18,49\},\{18,50\} \& actuator $G2$.
\item[Mode $q=3$] Line \{40,41\} \& actuator $G3$.
\item[Mode $q=4$] Lines \{18,49\},\{18,50\},\{27,53\},\{53,54\},\{60,61\} \& actuator $G4$.
\item[Mode $q=5$] Lines \{27,53\},\{40,41\},\{53,54\},\{60,61\} \& actuator $G5$.
\item[Mode $q=6$] Lines \{18,49\},\{18,50\},\{40,41\} \& actuator $G6$.
\item[Mode $q=7$] Lines \{18,49\},\{18,50\},\{27,53\},\{40,41\},\{53,54\},\{60,61\} \& actuator $G7$.
\item[Mode $q=8$] Actuator $G8$.
\end{description}

We consider a time-varying attack scenario where the attack mode is $q=2$ for $t = [0, 2.5)$s followed by  $q=5$ for $t = [2.5, 5)$s, while the actuator attack signal is given by $d_{a,i}=10^3 t$ for $t = [0, 1.25)$s, $d_{a,i}=10^3 (2.5-t)$ for $t = [1.25, 2.5)$s and $d_{a,i}=-500 (t-2.5)$ for $t = [2.5, 5)$s.
The goal of this case study is to demonstrate that our proposed approach can detect, identify and mitigate attacks. First, Figure \ref{fig:modeB} shows that the attacks are almost instantaneously detected and  the attack modes are quickly identified. In addition, Figure \ref{fig:attackB} shows that the actuator attack signal is successfully identified and similarly, all system states can be well estimated (not depicted for brevity). Finally, the attack mitigation scheme is successful at keeping the phase angles regulated to 10 rad/s despite attacks, whereas in the absence of attack mitigation, the phase angles can be significantly influenced by the attackers, as shown in Figure \ref{fig:statesB}.

\begin{figure}[!tp]
\begin{center} 
\begin{subfigure}[t]{0.49\textwidth}
        \centering
	\includegraphics[scale=0.575,trim=1mm 1mm 0mm 2mm,clip]{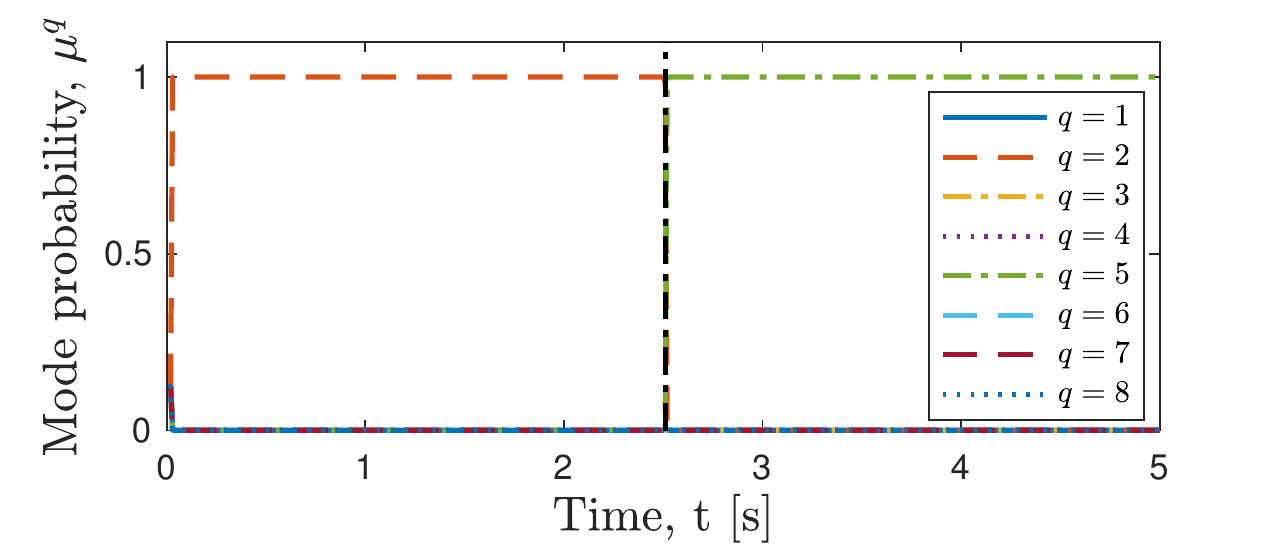}
	\caption{Posterior mode probabilities.\label{fig:modeB}}
\end{subfigure}
\begin{subfigure}[t]{0.49\textwidth}
	\centering
	\includegraphics[scale=0.575,trim=0mm 1mm 0mm 2mm,clip]{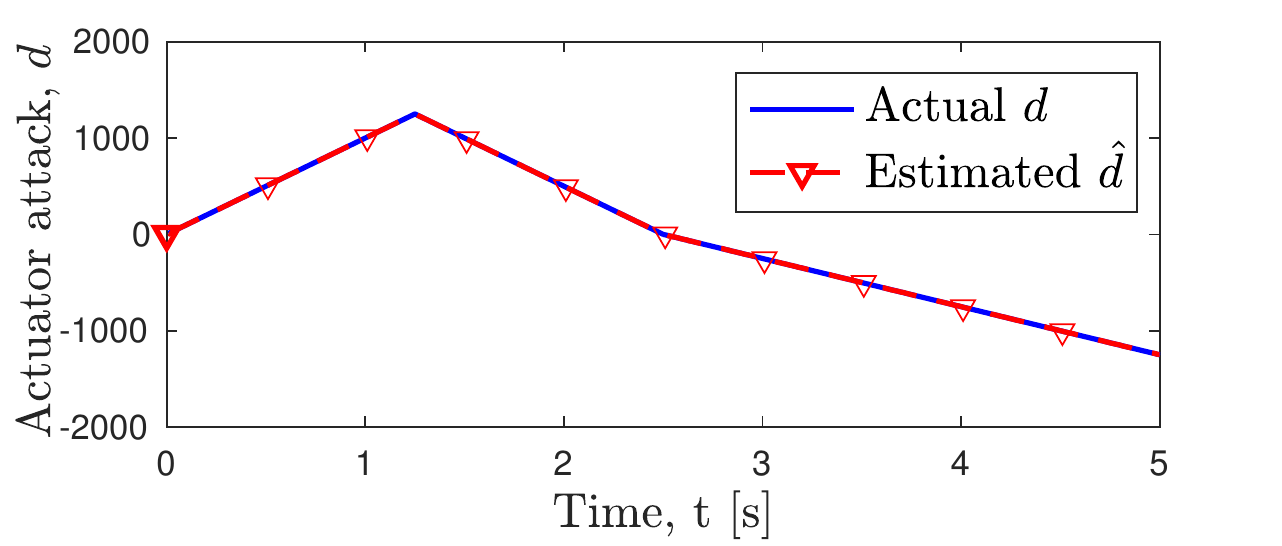}
	\caption{Actuator attack signal estimates.\label{fig:attackB}} 
\end{subfigure}
\vspace{-0.3cm}
\end{center}
\caption{Estimates of attack signal and mode probabilities when the attack mode switches from $q=2$ to $q=5$ at 2.5 s in Example \ref{eg1}.}
\end{figure}

\begin{figure}[!tp]
\begin{center} 
\includegraphics[scale=0.31,trim=35mm 5mm 0mm 0mm,clip]{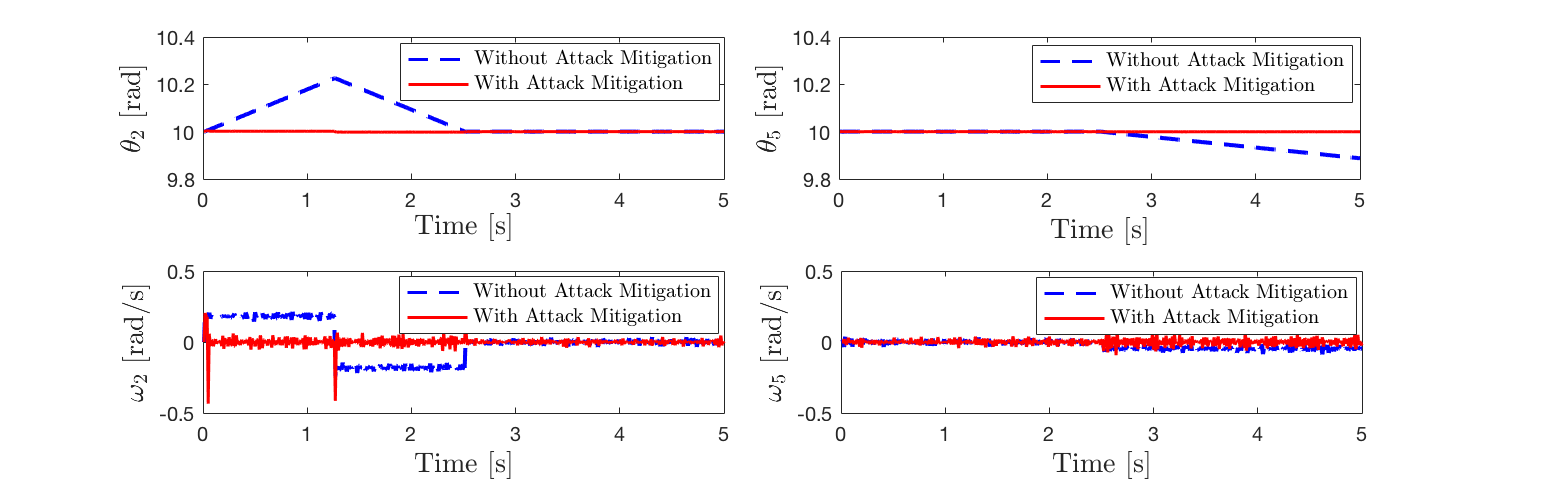}
\caption{A comparison of system states with and without the proposed attack mitigation.\label{fig:statesB}} \vspace{-0.3cm}
\end{center}
\end{figure}
}

\section{Conclusion} \label{sec:conclusion}

We addressed the problem of resilient state estimation for switching (mode/topology) attacks and  attacks on actuator and sensor signals of stochastic cyber-physical systems, which is especially important given the proliferation of the internet of things. We first modeled the problem as a hidden mode switched linear stochastic system with unknown inputs and showed that the multiple-model inference algorithm  in \cite{Yong.Zhu.ea.SICON16} is a suitable solution to these issues. Moreover, we provided an achievable upper bound on the maximum number of asymptotically correctable signal attacks and also the maximum number of required models for the multiple-model approach.  We also found sufficient conditions for attack (un-)detectability and attack identification, as well as designed an attack-mitigating feedback controller. %
Simulation examples, \yong{including one with the IEEE 68-bus test system,} demonstrated the effectiveness of our approach for resilient estimation and attack identification and mitigation.
\vspace{-0.05cm} 
\bibliographystyle{ACM-Reference-Format-Journals}
%\bibliography{acmsmall-sample-bibfile}
\bibliography{biblio}
                             % Sample .bib file with references that match those in
                             % the 'Specifications Document (V1.5)' as well containing
                             % 'legacy' bibs and bibs with 'alternate codings'.
                             % Gerry Murray - March 2012
                             
% Appendix
\appendix
\yong{\section*{APPENDIX}
\setcounter{section}{1}
\subsection{System Transformation} \label{sec:appendix}
%\vspace{-0.2cm}
To obtain the mode-matched input and state estimator \eqref{eq:d}, \eqref{eq:xstar} and \eqref{eq:stateEst}, we will consider a system transformation for the continuous system  dynamics and output equation in \eqref{eq:system} for each mode $q_k$ \cite{Yong.Zhu.ea.Automatica16}. First, we %perform a system transformation by decoupling the output equation involving $y_k$ into two components, one with a full rank direct feedthrough matrix and the other without direct feedthrough. %In this form, the filter can be designed leveraging existing approaches for both cases (e.g., \cite{Gillijns.2007,Yong.Zhu.Frazzoli.2013}). 
%Let $p_{H_k}:={\rm rk} (H_k)$. 
%Using singular value decomposition, we 
rewrite the direct feedthrough matrix $H_k$ using singular value decomposition as
%\begin{align}
$H_k= \begin{bmatrix}U_{1,k}& U_{2,k} \end{bmatrix} \begin{bmatrix} \Sigma_k & 0 \\ 0 & 0 \end{bmatrix} \begin{bmatrix} V_{1,k}^{\, \top} \\ V_{2,k}^{\, \top} \end{bmatrix}$, 
%\end{align}
where $\Sigma_k \in \mathbb{R}^{p_{H_k} \times p_{H_k}}$ is a diagonal matrix of full rank, $U_{1,k} \in \mathbb{R}^{l \times p_{H_k}}$, $U_{2,k} \in \mathbb{R}^{l \times (l-p_{H_k})}$, $V_{1,k} \in \mathbb{R}^{p \times p_{H_k}}$ and $V_{2,k} \in \mathbb{R}^{p \times (p-p_{H_k})}$ with $p_{H_k}:={\rm rk} (H_k)$, while $U_k:=\begin{bmatrix} U_{1,k} & U_{2,k} \end{bmatrix}$ and $V_k:=\begin{bmatrix} V_{1,k} & V_{2,k} \end{bmatrix}$ are unitary matrices. %Note that in the case 
When there is no direct feedthrough, $\Sigma_k$, $U_{1,k}$ and $V_{1,k}$ are empty matrices\footnote{We adopt the convention that the inverse of an empty matrix is also an empty matrix and assume that operations with empty matrices are possible.}, % These features are readily available in many simulation software products such as MATLAB, LabVIEW and GNU Octave. Otherwise, a conditional statement can be included to bypass this case.}, 
and $U_{2,k}$ and $V_{2,k}$ are arbitrary unitary matrices.

Moreover, we define two orthogonal components of the unknown input $d_k$ given by \vspace{-0.1cm}
\begin{align}
d_{1,k}=V_{1,k}^\top d_k, \quad
d_{2,k}=V_{2,k}^\top d_k.
\end{align}
Since $V_k$ is unitary, $d_k =V_{1,k} d_{1,k}+V_{2,k} d_{2,k}$. Thus, the continuous system  dynamics and output equation in \eqref{eq:system} for each mode $q_k$ can be rewritten as
\begin{align}
\nonumber x_{k+1}&=A_k x_k+B_k u_k+G_k V_{1,k} d_{1,k} +G_k V_{2,k} d_{2,k} +w_k\\
& = A_k x_k+B_k u_k+G_{1,k} d_{1,k} +G_{2,k} d_{2,k} +w_k, \ \, \label{eq:sysX}\\ %\end{align}\begin{align}
\nonumber y_k&=C_k x_k +D_k u_k + H_k V_{1,k} d_{1,k} + H_k V_{2,k} d_{2,k}  + v_k\\
&=C_k x_k +D_k u_k + H_{1,k} d_{1,k} + v_k, \label{eq:y}
\end{align}
where $G_{1,k} :=G_k V_{1,k}$, $G_{2,k} :=G_k V_{2,k}$ and $H_{1,k} :=H_k V_{1,k}=U_{1,k} \Sigma_k$. Next, we decouple the output $y_k$ 
using a nonsingular transformation $T_k =\begin{bmatrix} T_{1,k}^\top & T_{2,k}^\top \end{bmatrix}^\top$
\begin{align} \label{eq:T_k}
\begin{array}{ll} % \\
T_k &= \begin{bmatrix} I_{p_{H_k}} & -U_{1,k}^\top R_k U_{2,k} (U_{2,k}^\top R_k U_{2,k})^{-1}\\ 0 & I_{(l-p_{H_k}) } \end{bmatrix} \begin{bmatrix} U_{1,k}^\top \\ U_{2,k}^\top \end{bmatrix}\end{array}\hspace{-0.3cm}
\end{align}
to obtain $z_{1,k} \in \mathbb{R}^{p_{H_k}}$ and $z_{2,k} \in \mathbb{R}^{l-p_{H_k}}$ given by
\begin{align} \label{eq:sysY} \begin{array}{lll}
z_{1,k}&=T_{1,k} y_k &= C_{1,k} x_k + D_{1,k} u_k +\Sigma_k d_{1,k} + v_{1,k}\\
z_{2,k}&=T_{2,k} y_k &= C_{2,k} x_k + D_{2,k} u_k + v_{2,k}\end{array}
\end{align}
where $C_{1,k} :=T_{1,k} C_k$, $C_{2,k} := T_{2,k} C_k = U_{2,k}^\top C_k$, $D_{1,k} :=T_{1,k} D_k$, $D_{2,k} := T_{2,k} D_k = U_{2,k}^\top D_k$, $v_{1,k} :=T_{1,k} v_k$ and $v_{2,k} := T_{2,k} v_k = U_{2,k}^\top v_k$. This system transformation essentially decouples the output equation involving $y_k$ into two components, one with a full rank direct feedthrough matrix and the other without direct feedthrough. The transformation is also chosen such that the measurement noise terms for the decoupled outputs are uncorrelated. The covariances of $v_{1,k}$ and $v_{2,k}$ are %can then be found as follows:
\begin{align} \label{eq:R12}
\nonumber R_{1,k}&:=\mathbb{E}[v_{1,k} v_{1,k}^\top]=T_{1,k} R_k T_{1,k}^\top \succ 0,\\
\nonumber R_{2,k}&:=\mathbb{E}[v_{2,k} v_{2,k}^\top]=T_{2,k} R_k T_{2,k}^\top = U_{2,k}^\top R_k U_{2,k}  \succ 0, \ \\
R_{12,k}&:=\mathbb{E}[v_{1,k} v_{2,k}^\top] = T_{1,k} R_k T_{2,k}^\top = 0, \\% U_{1,k}^\top R_k U_{2,k} \\
% & \hspace{-0.75cm} - U_{1,k}^\top R_k U_{2,k} (U_{2,k}^\top R_k U_{2,k})^{-1}U_{2,k}^\top R_k U_{2,k}=0,  \\
 \nonumber R_{12,(k,i)}&:=\mathbb{E}[v_{1,k} v_{2,i}^\top] =T_{1,k}\mathbb{E}[v_{k} v_{i}^\top]T_{2,i}^\top=0, \ \forall k \neq i. %\in \mathbb{N}
%\nonumber & \hspace{-1.5cm}- U_{1,k}^\top R_k U_{2,k} (U_{2,k}^\top R_k U_{2,k})^{-1}U_{2,k}^\top R_k U_{2,k} =0, \; \forall k, i \in \mathbb{N}
 \end{align}
%Since the initial state and noise terms are assumed to be uncorrelated, %it is straightforward to verify that 
Moreover, $v_{1,k}$ and $v_{2,k}$ are uncorrelated with the initial state $x_0$ and process noise $w_k$.}

\end{document}